\documentclass[12pt]{article}
\usepackage{latexsym}
\usepackage{epsfig}
\usepackage{amsmath,amsthm,amssymb}
\usepackage{enumerate}

\usepackage[usenames]{color}
\usepackage{graphicx}

\definecolor{brown}{cmyk}{0, 0.72, 1, 0.45}
\definecolor{grey}{gray}{0.5}
\def\red{ \color{red}}

\newcounter{rot}

\newcommand{\bignote}[1]{\fbox{\parbox{6in}{\red #1}}%
\marginpar{\fbox{{\bf {\large Q \therot}}}}\addtocounter{rot}{1}}

\def\cN{{\cal N}}
\def\bcN{\overline{\cN}}

\def\cM{{\cal M}}

\def\a{\alpha} \def\b{\beta} \def\d{\delta} 
\def\e{\epsilon} \def\f{\phi} \def\F{{\Phi}}  \def\g{\gamma}
\def\G{\Gamma}  \def\k{\kappa}
 \def\th{\theta}  \def\Th{\Theta}  \def\l{\lambda}
   \def\p{\pi}
\def\r{\rho}  \def\s{\sigma} 
\def\t{\tau} \def\om{\omega}  \def\Om{\Omega}

\def\cG{{\cal G}}

\def\cB{{\cal B}}
\def\cT{{\cal T}}
\def\cS{{\cal S}}
\def\cR{{\cal R}}
\def\cH{{\cal H}}
\def\cU{{\cal U}}

\def\ve{\varepsilon}
\newtheorem{theorem}{Theorem}
\newtheorem{lemma}[theorem]{Lemma}
\newtheorem{corollary}[theorem]{Corollary}

\newcommand{\proofstart}{{\bf Proof\hspace{2em}}}

\newcommand{\proofend}{\hspace*{\fill}\mbox{$\Box$}}

\def\cX{{\cal X}}
\newcommand{\ooi}{(1+o(1))}
\newcommand{\sooi}{\mbox{${\scriptstyle (1+o(1))}$}}

\newcommand{\ul}[1]{\mbox{\boldmath$#1$}}
\newcommand{\ol}[1]{\overline{#1}}
\newcommand{\wh}[1]{\widehat{#1}}
\newcommand{\wt}[1]{\widetilde{#1}}

\newcommand{\rdown}[1]{{\left\lfloor #1\right \rfloor}}

\newcommand{\brac}[1]{\left(#1\right)}

\newcommand{\bfrac}[2]{\left(\frac{#1}{#2}\right)}

\def\cE{{\cal E}}
\newcommand{\rai}{\rightarrow \infty}
\newcommand{\ra}{\rightarrow}
\newcommand{\rat}{{\textstyle \ra}}

\newcommand{\set}[1]{\left\{#1\right\}}
\def\sm{\setminus}
\def\seq{\subseteq}
\def\es{\emptyset}

\def\E{\mbox{{\bf E}}}
\def\Var{\mbox{{\bf Var}}}
\def\Pr{\mbox{{\bf Pr}}}

\def\cF{{\cal F}}
\newcommand{\ignore}[1]{}

\newcommand{\cA}{{\cal A}}

\newcommand{\card}[1]{\left|#1\right|}

\newcommand{\beq}[1]{\begin{equation}\label{#1}}
\newcommand{\eeq}{\end{equation}}

\def\cN{{\cal N}}
\def\dist{\;\text{dist}}

\def\ole{o(1)}
\def\bd{{\bf d}}
\def\cH{{\cal H}}

\def\Tc{T_{\mbox{cov}}}

\begin{document}

\makeatletter
\title{Vacant sets and vacant nets: Component structures
   induced by a random walk.}
\author{Colin Cooper\thanks{Department of  Informatics,
King's College, University of London, London WC2R 2LS, UK.
Research supported in part by EPSRC grants EP/J006300/1 and EP/M005038/1.}
\and Alan
Frieze\thanks{Department of Mathematical Sciences, Carnegie Mellon
University, Pittsburgh PA 15213, USA.
Research supported in part by NSF grant DMS0753472.
}
}
\maketitle \makeatother
\begin{abstract}

Given a discrete random walk on a finite graph $G$,
the vacant set and vacant net are, respectively, the sets of vertices and  edges
which remain unvisited by the walk at a given  step $t$.
Let $\G(t)$ be the subgraph of $G$ induced by the vacant set of the walk
 at step $t$. Similarly, let $\wh \G(t)$ be the  subgraph of $G$
induced by the  edges of the vacant net.

For random
$r$-regular graphs $G_r$, it was previously established that
for a simple random walk,
the graph $\G(t)$ of the vacant set undergoes
 a phase transition in the
sense of the
phase transition on Erd\H{os}-Renyi graphs $G_{n,p}$.
Thus, for $r \ge 3$  there is an explicit value $t^*=t^*(r)$ of the walk, such that
for $t\leq (1-\e)t^*$, $\G(t)$ has a unique giant component, plus
components of  size $O(\log n)$, whereas  for
$t\geq (1+\e)t^*$ all the components of $\G(t)$ are of size $O(\log n)$.

In this paper we  establish the  threshold value  $\wh t$ for a
phase transition in the
graph $\wh \G(t)$
of the vacant net of a simple random walk on a random $r$-regular graph,.

We   obtain the corresponding threshold results
 for the vacant set  and  vacant net of
two modified random walks.
These are  a non-backtracking random walk, and, for $r$ even, a random walk which
chooses unvisited edges whenever available.

This allows a direct comparison of thresholds between
 simple and modified
walks on random $r$-regular
graphs. The main findings are the following: As $r$ increases
the threshold for the vacant set converges to $n \log r$ in all three walks.
 For the vacant net,  the threshold  converges to $rn/2 \; \log n$ for both the simple random walk and non-backtracking random walk.
When $r\ge 4$ is even,
the threshold for the vacant net
of the unvisited edge process  converges to $rn/2$, which is also the vertex cover time of the process.
\end{abstract}

\section{Introduction}

Let $G=(V,E)$ be a finite connected graph, with vertex set size $|V|=n$, and edge set size $|E|=m$.
Let $W$ be a simple random walk on $G$, with initial position $X(0)$ at $t=0$.
At discrete steps $t=1,2,\cdots$, the walk chooses  $X(t)$
uniformly at random (u.a.r.) from the neighbours of $X(t-1)$ and makes the edge
transition $(X(t-1), X(t))$.
Let $W(t)=(X(0),...,X(t))$ be the trajectory of the walk
up to and including step $t$, and let $\cB(t)=\set{X(s): s \le t}$
be the set of vertices visited in $W(t)$.
By analogy with site percolation, the set of unvisited vertices $\cR(t)=V \sm \cB(t)$ is
referred to as the {\em  vacant set} of the walk. The graph induced by the uncrossed edges is referred to as the {\em vacant net}.

In the case of random $r$-regular graphs, it was established independently by \cite{CTD} and \cite{CFVac}
that the graph induced by the set of unvisited vertices exhibits sharp threshold behavior. Typically,
as the walk proceeds,  the induced graph of the vacant set has a unique giant component, which  collapses within a relatively small
number of steps to leave components of at most logarithmic size.
For random $r$-regular graphs,  we establish the threshold behavior
of the  vacant net, i.e. the subgraph induced by the set of unvisited edges of the random walk.
For comparison purposes, and ignoring terms of order $1/r$, the thresholds for the vacant set and vacant net
occur around steps $n \log r$ and $(r/2) n \log r$
 of the walk, respectively.

For $v\in V$ let $C_v$ be the expected time taken for a
 random walk $W_v$ starting at vertex $X(0)=v$, to visit every vertex of the graph
$G$.  The {\em vertex cover time} $\Tc^V(G)$ of a graph $G$ is
defined as $\Tc^V(G)=\max_{v\in V}C_v$.
Let $N(t)= |\cR(t)|$ be the size of the vacant set at step $t$ of the walk.
As the walk $W_v(t)$ proceeds, the size of the
vacant set decreases from $N(0)=n$ to $N(t)=0$ at expected time $C_v$.
The change in structure of the graph $\G(t)=G[\cR(t)]$ induced by the vacant set
$\cR(t)$ is also of interest, insomuch as it is
reasonable to ask if $\G(t)$ evolves in a typical way for most walks $W(t)$.
Perhaps surprisingly the component structure of the vacant set can be described in
detail for certain types of  random graphs, and also to some extent for toroidal grids of dimension
at least 5.

To motivate this description of the component structure,
we recall the typical evolution of the random graph
$G_{n,p}$ as $p$ increases from 0 to 1. Initially, at $p=0$, $G_{n,0}$ consists of isolated vertices.
As we increase $p$, we find that for $p=c/n$, when $c<1$ the maximum
component size is
logarithmic. This is followed by a phase transition around the critical value $c=1$.
When $c>1$ the  maximum
component size is linear in $n$, and all other components have logarithmic
size.

In describing the evolution of the structure of the vacant set as $t$ increases, the aim is to show
that typically
 $\G(t)$ undergoes a reversal of the phase transition mentioned above. Thus $\G(0)$
is connected and $\G(t)$  starts to break up as $t$ increases.
There is a critical value $t^*$ such that if $t<t^*$ by a sufficient
amount then $\G(t)$ consists of a unique giant
component plus components of size $O(\log n)$. Once we
 pass through the critical value
 by a sufficient amount, so that $t>t^*$,
then all components are of size $O(\log n)$. As $t$ increases further, the maximum component
size  shrinks to zero.
We make the following definitions. A graph with vertex set $V_1$ is {\em sub-critical} if its maximum
component size is $O(\log n)$, and
{\em super-critical} if is has a {\em unique}
component $C_1(t)$ of size $\Omega(|V_1(t)|)$, where $|V_1(t)|\gg\log n$,
and
all other components are of size $O(\log n)$.

For the case of random $r$-regular graphs $G_r$, the vacant set
was studied independently by \v{C}erny, Teixeira and Windisch \cite{CTD}
and by Cooper and Frieze \cite{CFVac}. Both
\cite{CTD} and \cite{CFVac} proved that w.h.p. $\G(t)$ is sub-critical for $t\geq (1+\e)t^*$
and that there is a unique linear size component
for $t\leq (1-\e)t^*$. The paper \cite{CTD} conjectured that $\G(t)$ is super-critical
for $t\leq (1-\e)t^*$, and this was confirmed by \cite{CFVac}
who also gave the detailed structure of the small ($O(\log n))$ tree components as a function of $t$.
Subsequent to this  \v{C}erny and
Teixeira \cite{CT2}  used the
methods of \cite{CFVac} to give a sharper analysis of $\G(t)$ in the critical window around $t^*$.
The  paper
\cite{CFVac}, also established the critical value $t^*$
for connected random graphs $G_{n,p}$ and for strongly connected
random digraphs $D_{n,p}$.

For the case of toroidal grids, the situation is less clear.
Benjamini and Sznitman \cite{BS} and  Windisch \cite{DW}
investigated the structure of the vacant set of a random walk
on a $d$-dimensional torus. The main focus of this work is to apply the method of
random interlacements. For toroidal grids of dimension $d \ge 5$, it is shown that
there is a value $t^+(d)$,  linear in $n$, above which the vacant set is sub-critical, and a value of
$t^-(d)$ below which the graph is super-critical. It is believed that there is
 a phase transition for $d \ge 3$.
A recent monograph by \v{C}erny and Teixeira \cite{CT}
summarizes the random interlacement methodology. The monograph also gives details for the
vacant set of random $r$-regular graphs.

Let $\cS(t)=\set{ (X(s),X(s+1)): 0 \le s <t}$ be the set of {\em visited edges}
 based on transitions of the walk $W$ up to and including step $t$, and let
  $\cU(t)=E(G) \sm  \cS(t)$ be corresponding the set of {\em unvisited edges}.
  The {\em edge cover time} $\Tc^E(G)$ of a graph $G$ is defined in a similar way
  to the vertex cover time.
The edge set $\cU(t)$ defines an edge induced subgraph $\wh \G(t)$ of $G$
whose vertices may be either visited or unvisited. By analogy with the case
for vertices we will call $\wh\G(t)$ the {\em vacant network} or {\em vacant net} for short.
We can ask the same questions about the phase transition $\wh t$ for the vacant net,
as were asked for the phase transition $t^*$ of the vacant set.

Random walk based crawling is a simple method to search large networks, and  a giant component in the
vacant set can indicate the existence of a large corpus of information which has somehow been missed.
Similarly, a giant component in the vacant net indicates the
continuing existence of a large communications network or set of
unexplored relationships.
From this point of view, any way to speed up
 the collapse of the giant component
can be seen as worthwhile.
One  method, which seems attractive at first sight,
is to prevent the walk from backtracking
over the edge it has just used. Another simple method is to walk randomly but choose
unvisited edges when available.

We determine the thresholds for simple random walks and non-backtracking random walks; and also for  walks which prefer unvisited edges for the case that the vertex degree $r$ is even.
This allows a direct comparison of performance between
these three types of random walk.
Detailed definitions and results for simple random walks, non-backtracking walks, and walks which prefer unvisited edges are given in Sections \ref{simpl}, \ref{no-o-bak} and \ref{unviz}
respectively.

As an example, for random 3-regular graphs, using a non-backtracking walk
reduces the threshold value by a factor of 2 for vacant sets, and by 5/2 for
 vacant nets respectively.
Thus for very sparse graphs, improvements can be obtained by making the walk non-backtracking.
However, the improvement gained by a non-backtracking walk
is of order $1+O(1/r)$, and soon becomes insignificant as $r$
increases. In fact, for all three walks, the threshold value for the vacant set tends to $n \log r$. For the vacant net, the threshold value tends to $nr/2 \log r$ for simple and non-backtracking walks.
For walks which prefer unvisited edges the threshold for the vacant net tends to $nr/2$, but the results only hold for $r$ even.

As a by-product of the proofs in this paper we
give an asymptotic value of  $(r/2)n$ for the vertex  cover time
of the unvisited edge process for $r$ even. This confirms the
order of magnitude estimate $\Th(n)$ and the constant $r/2$ in the
experimental results of \cite{BCF}. The plot of experiments is reproduced in Section \ref{expt} of the Appendix. Note that the plot uses the notation $d$ for vertex degree (rather than $r$).
It can be seen from the figure that the vertex cover time of the  unvisited edge process  exhibit a dichotomy whereby for odd vertex degree,  the vertex cover time appears to be $\Th(n \log n)$.

{\bf Notation.}\\
Apart from $O(\cdot),o(\cdot),\Om(\cdot)$  as a  function of $n \rai$, where $n=|V|$,
we use the following notation.
We say $A_n\ll B_n$ or $B_n\gg A_n$ if $A_n/B_n\to 0$ as
$n\to\infty$, and $A_n \sim B_n$ if $\lim_{ n \rai} A_n/B_n=1$.
The notation $\om(n)$ describes a function tending to infinity as $n \rai$.
We measure both walk and graph probabilities in terms of $n$,
 the size of the vertex
set of the graph.

We use the expression {\em with high probability} (w.h.p.),  to mean with probability
$1-o(1)$, where the $o(1)$ is a function of $n$,
which tends to zero as $n \rai$. For the proofs in this paper,
we can take $o(1)=O(\log^{-K}n)$ for some large positive constant $K$.
The statement of theorems in this section  are w.h.p. relative to both graph sampling and walks on the sampled graph.
It will be clear when we are discussing
properties of the the graph,  these are given in Section \ref{props}.
In the case where we use deferred decisions, if $|\cR(t)|=N$, the
w.h.p. statements are asymptotic in  $N$,
and we assume $N(n) \rai$ with $n$.

Let $W$ be a random walk $W$ on a graph $G$.
If we need to stress the start position $u$ of the walk $W$, we write $W_u$.
The vertex occupied by $W$ at step $t$ is given by $X(t)$ or $X_u(t)$.
Generally we use $\Pr(A)$ or $\Pr_W(A)$ to denote the probability of event $A=A(t)$ at some step $t$ of the random walk $W$.
We use $P$ for the transition matrix of the walk, and use $P_u^{t}(v)$ or $P_u^{t}(v; G)$  for the $(u,v)$-th entry of $P^t$, i.e
 $ P_u^{t}(v)=\Pr(X_u(t)=v)$.  When using  generating functions we use simple unencumbered notation such as $h_t,f_t,r_t$ for the probability that certain specific events occur at step $t$. In particular for a designated start vertex $v$,
$r_t=P_v^{t}(v)$. We use $\pi_v$, or $\pi_G(v)$ for the stationary probability of a random walk $W$ at vertex $v$ of a graph $G$.
The notation $p_v$ has a specific meaning in the context of Lemma \ref{First},
and is reserved for that.

\subsection{
Simple random walk: Structure of vacant set and vacant net}
\label{simpl}

Let $G_r(n)$ be the space of $r$-regular graphs on $n$ vertices,
and let $G$ be chosen u.a.r. from $G_r(n)$.
The following theorem details established
results for the vacant set of a simple random walk on $G$, as given in \cite{CTD}, \cite{CFVac}.

\begin{theorem}\label{oldstuff}
Let $W(t)$ be a simple random walk on a random $r$-regular graph.
For $r \ge 3$, the following results hold w.h.p..
\begin{enumerate}[(i)]
\item Let $\G(t)$ be the graph  induced by the vacant set $\cR(t)$, at step $t$ of $W$, then
 $\\G(t)$ has $|\cR(t)|$ vertices and $|E(\G(t))|$ edges, where
\beq{Nt}
|\cR(t)| \sim n \exp\brac{-\frac{r-2}{r-1}\frac{t}{n}},\qquad
|E(\G(t))| \sim \frac{rn}{2} \exp\brac{-\frac{2(r-2)}{r}\frac{t}{n}}.
\eeq
\item
The size of
the vacant
net $|\cU(t)|$ at step $t$ of $W$ is
\beq{Ut}
 |\cU(t)| \sim
\frac{rn}{2} \exp\brac{-\frac{2(r-2)}{r(r-1)}\frac{t}{n}}.
\eeq
\item \cite{CFReg}  The vertex and
edge cover times of a non-backtracking walk
are $\Tc^V(G)\sim \frac{r-1}{r-2} n \log n$ and $\Tc^E(G) \sim  \frac{r(r-1)}{2(r-2)} n \log n$ respectively.
\item  The threshold for
the sub-critical phase of the vacant set in $G$ occurs at
$t^*= u^* n$ where
\beq{tstar}
u^* = \frac{r(r-1)}{(r-2)^2} \log (r-1).
\eeq
\end{enumerate}
\end{theorem}

We now come to the new results of this paper.
We first consider  the structure of the graph $\wh \G(t)$ induced by the edges in the vacant net $\cU(t)$ of $G_{r}$.
 By using the random walk to
  reveal the structure of the graph, we argued in \cite{CFVac} that $\G(t)$ was a
  random graph with degree sequence $D_s(t), s=1,...,r$. We applied the result of
   Molloy and Reed \cite{MR1} for the existence of a giant component in fixed degree sequence graphs, to  the degree sequence $D_s(t)$  to obtain the threshold $t^*=u^*n$ given in \eqref{tstar}.
By using a simplification of the Molloy-Reed condition in terms of moments of the degree sequence we  can obtain the threshold  for
 the vacant net $\wh \G(t)$. The proof of the next theorem is given in Section
\ref{proofs}.
\begin{theorem}\label{th3}\label{SRW}
 Let $\wh t=\th^* n$. Then w.h.p. for any $\e >0$,
the graph $\wh \G(t)=(V, \cU(t))$ induced by the unvisited edges  $\cU(t)$ of $G$ has the following
properties:
\begin{enumerate}[(i)]
\item  The threshold for
the sub-critical phase of the vacant net in $G$ occurs at
$\wh t= \th^* n$ where
\beq{t*}
\th^*= \frac{r(r^2-2r+2)}{2(r-2)^2} \log (r-1).
\eeq
\item For $t\leq (1-\e)\wh t$, $\wh\G(t)$
is super-critical,  and $|C_1(t)|=\Om(n)$.
\item For $t\geq (1+\e) \wh t$, $\wh \G(t)$ is sub-critical, and thus $|C_1(t)|=O(\log n)$.
\item For some constant $c>0$ and  $t \in (\wh t -cn^{2/3}, \wh t + cn^{2/3})$, then $\Pr(|C_1(t)|= \Th(n^{2/3})) \ge 1-\e$.
\end{enumerate}
\end{theorem}

\subsection{
Non-backtracking random walk: Structure of vacant set and vacant net}
\label{no-o-bak}

Speeding up random walks is a matter of both theoretical curiosity  and practical interest. One plausible approach to this
is to
use a non-backtracking walk. A  non-backtracking walk does not
move back down the edge used for the previous transition unless there is no choice. Thus arguably it  should be faster to cover the graph.
Let $v=X(t)$ be the vertex occupied by the walk at step $t$, and suppose this vertex was reached by
the edge transition  $e=(X(t-1),X(t))$. The vertex $u=X(t+1)$
 is chosen u.a.r. from $N(v) \sm X(t-1)$, so that
 $e \ne (X(t),X(t+1))$.
If there is no choice, i.e.
$X(t)$ is a vertex of degree 1,
we can assume the walk returns along $e$, but as $r \ge 3$ this
case does not arise.

In the case of random $r$-regular graphs, a direct comparison  can be made
between the performance of simple and non-backtracking random walks.
The details for non-backtracking walks are summarized in the following theorem,
the proof of which is given in Section\ref{noback}.
The comparable results for simple walks are given in Section \ref{simpl}.

\begin{theorem}\label{nobac}
Let $W(t)$ be a non-backtracking random walk on a random $r$-regular graph.
For $r \ge 3$,  the following results hold w.h.p..
\begin{enumerate}[(i)]
\item Let $\G(t)$ be the graph  induced by the vacant set $\cR(t)$, at step $t$ of $W$, then
 $\\G(t)$ has $|\cR(t)|$ vertices and $|E(\G(t))|$ edges, where
\[
|\cR(t)| \sim  n  \exp\brac{-t/n},
\qquad\qquad
|E(\G(t))| \sim \frac{rn}{2} \exp\brac{-\frac{2(r-1)t}{rn}}.
\]
\item
The size of
the vacant
net $|\cU(t)|$ at step $t$ of $W$ is
\[
|\cU(t)| \sim  \frac{rn}{2}  \exp\brac{-2t/rn}.
\]
\item The vertex and
edge cover times of a non-backtracking walk
are $\Tc^V(G)\sim n \log n$ and $\Tc^E(G) \sim  (r/2) n \log n$ respectively.
\item
The threshold  for the sub-critical phase of the vacant set occurs at $t^*=u^* n$ where
\[
u^* \sim  \frac{r}{r-2} \log (r-1).
\]
\item
The threshold  for the sub-critical phase of the vacant net occurs at $\wh t=\th^* n$ where
\[
 \th^* \sim  \frac{r(r-1)}{2(r-2)} \log (r-1).
\]
\item Let $\wt t=t^*, \wt t$, for the vacant set and vacant net respectively.
    For any $\e >0$, some constant $c>0$ and   $t \in (\wt t -cn^{2/3}, \wt t + cn^{2/3})$, then $\Pr(|C_1(t)|= \Th(n^{2/3})) \ge 1-\e$.
\end{enumerate}
\end{theorem}
Comparing $u^*, \th^*$ for simple and non-backtracking walks, from \eqref{tstar}, \eqref{t*}
and Theorem \ref{nobac} respectively, we see that for $r=3$ the subcritical phases occur $2,\; 5/2$
times  earlier for vacant sets and vacant nets (resp.). This improvement decreases  rapidly as $r$ increases. A direct contrast between the densities of the vacant set for the two walks follows from the edge-vertex ratios $|E(\G(t))|/|\cR(t)|$.
At any step $t$ the vacant set of the simple random walk is denser w.h.p..

\subsection{Random walks which prefer unvisited edges: Structure of vacant set and vacant net}\label{unviz}

The papers \cite{BCF}, \cite{OS} describe a modified random  walk $X=(X(t),\; t \ge 0)$ on a graph $G$, which uses unvisited edges when available at the currently occupied vertex.
If there are {\em unvisited  edges} incident with the current
vertex, the walk  picks one u.a.r. and make a transition along this edge.
If there are no unvisited edges incident with the current vertex, the walk moves to a random neighbour.

In \cite{BCF} this walk
was called an {\em unvisited edge process} (or  edge-process), and in \cite{OS},
a {\em greedy random walk}.
For random  $r$-regular graphs where $r=2d$, it was shown in \cite{BCF} that the edge-process has vertex cover time $\Th(n)$, which is best possible up to a constant. The paper also gives an upper bound of  $O(n \om)$ for the edge cover time. The $\om$ term comes from the w.h.p. presence of small cycles (of length at most $\om$).

In the case of random $r$-regular graphs,
the vacant set and vacant net of the edge-process
have the following theorem which is proved in Section \ref{edge-proc}.

\begin{theorem}\label{newedge}
Let $X$ be an edge-process on a random $r$-regular graph.
For $r \ge 4, \; r=2d$,
the following results hold w.h.p..
\begin{enumerate}[(i)]
\item Let $\G(t)$ be the graph  induced by the vacant set $\cR(t)$ of the edge-process at step $t$. Then for $\d>0$ and any $t=dt(1-\d)$ the vacant set
 has $|\cR(t)|$ vertices and $|E(\G(t))|$ edges, where
\[
|\cR(t)| \sim  n \bfrac{dn-t}{dn}^d,
\qquad\qquad
|E(\G(t))| \sim  dn \bfrac{dn-t}{dn}^{2d-1}
\]
\item The vertex cover time of the edge-process is
$\Tc^V(G)\sim dn$.
\item
The threshold  for the sub-critical phase of the vacant set occurs at $t^*\sim u^* n$ where
\[
u^* \sim  d \brac{1- \bfrac{1}{2d-1}^{\frac{1}{d-1}}}.
\]
For any $\e>0$ and $t=t^*(1-\e)$, the largest component $C_1(t)$ is of size $\Th(n)$, whereas for $t=t^*(1+\e)$, the largest component is of size $O(\log n)$.
\item
For $t = dn(1-\d)$, and $\d \ge \sqrt{\om \log n/n}$, the
the vacant
net $\cU(t)$  of the edge-process is of size $dn\d\ooi$.
\item
The threshold  for a phase transition of the vacant net occurs at $\wh t \sim dn$.
For any $\e>0$ and $t=\wh t(1-\e)$, the largest component $C_1(t)$ is of size $\Th(n)$, whereas for $t=\wh t(1+\e)$, the largest component is of size $O(\log n)$.
\end{enumerate}
\end{theorem}
As for the edge cover time $\Tc^E(G)$ of the edge-process, trivially
$\Tc^E(G) \ge dn$.
It was proved in \cite{BCF} that  $\Tc^E(G)=O(\om n)$. The $\om$ term comes from
the presence of cycles size $O(\om)$. We do not see any obvious reason
from the proof of Theorem \ref{newedge} to suppose $\Tc^E(G)=\Th(n)$.

\subsection{Outline of proof methodology}

The proof of the vacant net threshold, Theorem \ref{SRW}, is given in Section \ref{proofs}. The proof of Theorem \ref{nobac} on the properties of the vacant set and vacant net for non-backtracking random walks is given in Section \ref{noback}. The technique  used to analyze the structure of random walks is one  the authors have developed over a sequence of papers. The results we need in the proof of this paper are given in Section \ref{basiclem}.

The method of proof of the main theorems is similar.
The main steps in the  proof of (e.g.) Theorem \ref{SRW}
are as follows.
(i) In Section \ref{props} we state the structural graph properties
we assume in order  to analyse a random walk on an $r$--regular graph.
(ii) Given these properties, in Section \ref{degseq} we obtain the degree sequence $\wh \bd(t)$ of the
vacant net $\wh \G(t)$ at step $t$ of the walk. The degree sequence is given in an implicit form.
(iii) In Section \ref{uniformity}, we prove that $\wh \G(t)$ is a random graph with degree sequence $\wh \bd(t)$.
(iii) In Section \ref{MRC} we obtain the component structure of $\wh \G(t)$.
This follows from a result of Molloy and Reed
\cite{MR1} on the component structure of fixed degree sequence random graphs.

We next give more detail of the general method used  to prove structural
properties of the vacant set or vacant net. For ease of description we use the example of the vacant set of a simple random walk, and highlight any differences for the other cases as appropriate.
There are two main features.

Firstly
we use the random walk to generate the graph in the configuration model.
If we stop the walk at any step, the un-revealed part of the graph is still
random conditional on the structure of the revealed part, and the
constraint that all vertices have degree $r$. The approach is equally valid for
other Markov processes such as non-backtracking random walks.
Secondly using the techniques given in Section \ref{unvisit} we can
estimate the size $N(t)$, and degree sequence $\bd(t)$, of the vacant set $\cR(t)$ very precisely at a given step $t$.

Combining these results, the graph $\G(t)$ of the vacant set
is thus a random graph with $N(t)$ vertices and degree sequence $\bd(t)$.
Molloy and Reed \cite{MR1} derived  conditions for the existence of, and size of the giant component in a random graph with a given degree sequence. We apply these conditions to $\G(t)$ to obtain the threshold etc. This is what we did in \cite{CFVac}, and we do not reproduce in detail those aspects of (e.g.) Theorem \ref{nobac} which directly repeat these methods.

\ignore{
Using the same approach, we can also
estimate the degree sequence $\wh {\bd}(t)$, and size $\wh N(t)$  of the vacant net $\wh \G(t)$, the subgraph induced by the unvisited edges. The answers are obtained as
 alternating sums, which make finding an explicit threshold for the vacant net
using the Molloy-Reed condition in Theorem \ref{MR}
seem untractable.
The statement of the Molloy-Reed threshold condition can however be simplified,
in as much as it can be rephrased in terms of some  moments $M(1,t),\; M(2,t)$ of the random walk which we define in Lemma \ref{th4}.
This allows us to apply existing methods to obtain the threshold for the vacant net.
}

\section{Graph properties of $G_r$}\label{props}

Let
\beq{ell1}
\ell_1=\e_1\log_rn,
\eeq
for some sufficiently small $\e_1$. A cycle $C$ is {\em small} if $|C|\leq \ell_1$.
A vertex of a graph $G$ is {\em nice} if it is at distance at least $\ell_1+1$ from any small cycle.

Let $D_{k}(v)$ be the subgraph of $G$ induced by the vertices at distance at most $k$ from $v$.
A vertex $v$ is {\em tree-like} to depth $k$ if $D_k(v)$ induces a tree, rooted at $v$.
Thus a nice vertex is tree-like to depth $\ell_1$.
Let $\cN$ denote the nice vertices of $G$ and $\bcN$ denote the vertices that are not nice.

Let $G_r$ be the space of $r$-regular graphs, endowed with the uniform probability measure. Let $G$ be chosen u.a.r. from $G_r$.
We assume the following w.h.p.  properties.
\begin{align}
&\text{There are at most $n^{2\e_1}$ vertices that are not nice.}\label{nice}\\
&\text{There are no two small cycles within distance $2\ell_1$ of each other.}\label{nice1}\\
& \text{Let } \l= \max(\l_2,\l_n) \text{ be the second largest eigenvalue of the
transition matrix } P. \nonumber\\
& \text{Then } \l_2\leq (2\sqrt{r-1}+\e)/r\leq 29/30, \text{ say.} \label{nice2}
\end{align}
Properties (i), (ii) are straightforward to prove by first moment calculations.
Property (iii) is a result of Friedman \cite{Fried}.

The results we prove concerning random walks on a graph $G$ are all conditional on $G$ having properties \eqref{nice}-\eqref{nice2}.
This conditioning can only inflate the probabilities of unlikely events by $1+o(1)$. This observation includes those
events defined in terms of the configuration model as claimed in Lemma \ref{walkconfig}.
For $r$ constant, the underlying configuration multi-graph is simple with constant probability, and all simple $r$-regular graphs are equally probable.
If a calculation shows that an event $\cE$ has probability at most $\e$ in
the configuration model, then it has probability $O(\e)$ with respect to the corresponding simple graph $G$. We only need to multiply this bound by
a further $1+o(1)$ in order to
estimate the probability conditional on \eqref{nice}-\eqref{nice2}. We will continue using  this convention without further comment.

\section{Background material on unvisit probabilities}\label{basiclem}
\subsection{Summary of methodology}
To find the size of the vacant set or net, we estimate the probability that a given vertex or edge
of  the graph   were  not visited by the random walk during steps $T,..,t$,
 where $T$ is suitably defined mixing time (see \eqref{4}).
 For simplicity, we refer to this quantity as an {\em unvisit probability}.
 We briefly outline of how the unvisit probability is obtained.
This is given in more detail in Section \ref{unvisit}.

 The  quantities needed to estimate the unvisit probability of a vertex $v$
 are the mixing time $T$, the stationary probability $\pi_v$ of  vertex  $v$
 and  $R_v$, defined below.  For a simple random walk $\pi_v=d(v)/2m$.
The  mixing time $T$ we use satisfies a convergence condition given in \eqref{4}.
The theorems in this paper are for random regular graphs $G=G_r$, $r \ge 3$ constant, and w.h.p.  $G$
has constant eigenvalue gap so the mixing time $T=O(\log n)$ satisfies
\eqref{4}.
The non-backtracking walk uses a Markov chain $\cM$  on directed edges.
In Section \ref{snotproof} of the Appendix we prove directly that w.h.p. $T=O(\log n)$

The  unvisit probability $\Pr_W(\cA_v(t))$ is given in  \eqref{atv}-\eqref{atv1} of Corollary \ref{geom} of Lemma \ref{First}  in terms of  $p_v=\sooi \pi_v/R_v$.
For regular graphs $\pi_v=1/n$. The quantity $R_v$ is defined as follows.
For a walk starting from $v$ let $r_0=1$
 and let $r_i$ be the probability the walk returns to $v$ at step $i$. Then
 \[
 R_v= \sum_{i=0}^{T-1} r_i.
 \]
 Thus $R_v$ is the expected number of returns to
 $v$ before step $T$.

Because the Molloy-Reed condition is robust to small changes in degree sequence, for our proofs,
we only need to find the  value of $R_v$ for nice vertices.
This is obtained as follows.
Let $D_{\ell}(v)$ be the subgraph induced by
the vertices at distance at most $\ell$ from $v$. The value of $\ell$ we use is given in
\eqref{ell1}. If $D_\ell(v)$ is a tree,
 we say $v$ is a nice vertex, and use
  $\cN$ to denote the set of nice vertices of graph $G$.
With high probability,  all but $o(n)$ vertices of a random $r$-regular graph are nice.
If $v$ is nice,  the subgraph $D(v)$ is a tree with internal vertices of degree $r$, and we extend $D_\ell(v)$  to an infinite $r$-regular tree $\cT$
rooted at $v$.
The principal quantity used to  calculate $R_v$,
 is $f$, the probability of a first return to $v$ in $\cT$.
 Basically, once the walk is distance $\Th(\log \log n)$ from $v$
 the probability of a return to $v$ during
 $T=O(\log n)$ steps is $o(1)$. Thus calculations for $f$ can be made
 in $\cT$ followed by a correction of smaller order, giving $R_v=(1+o(1))/(1-f)$.
This is formalized in Lemma \ref{lemRv} of the Appendix.

The proofs in this paper use the notion of a set $S$ of vertices or edges
not being visited by the walk during $T,...,t$. Because $R_S$ is not well defined for general sets $S$, to
use Corollary \ref{geom} we contract the set  $S$
to a single vertex $\g(S)$, and
calculate
$R_{\g(S)}$ in the multi-graph $H$ obtained from $G$ by this
contraction. Using Corollary \ref{geom} we obtain the probability that $\g(S)$ is unvisited in $H$.
Lemma \ref{contract}
ensures that the probability  $\g(S)$ is unvisited in $H$ is asymptotically equal to the probability the set $S$ in unvisited in $G$.
In the case of visits to sets of edges rather than vertices, these are subdivided by inserting a
set of dummy vertices $S$, one in the
middle of each edge in question. The set $S$ is then contracted to a vertex $\g(S)$ as before. In the case of the non-backtracking walk things get more complicated as the Markov chain $\cM$ of the walk is on directed edges, but the principle is the same.

The contraction operation changes the graph from $G$ to $H$, which can
alter the mixing time $T$, but  does not
significantly increase it for the following reasons.
The effect of contracting a set of vertices
increases the eigenvalue gap, (see e.g. \cite{LPW} page 168)
so that $1-\l_2(H) \ge 1-\l_2(G)$, and thus $T$
can only decrease.
In the case of edge subdivision, the gap could decrease. However, we only
perform this operation on (at most) $2r$  edges of an $r$-regular graph
with constant eigenvalue gap, and with $r$ constant.
It follows  that the  conductance of $H$
is still constant and thus the mixing time $T(H)$
differs from $T(G)$ by at most a constant multiple.

\subsection{Unvisit probabilities}\label{unvisit}

Our proofs make heavy use of Lemma \ref{First} below.
Let $P$ be the transition matrix of the walk and let
$P_{u}^{t}(v)$ be the $(u,v)$--th entry of $P^t$.
Let $W_u(t)$ be the position of the random walk $W_u$ at step $t$, and  let
$P_{u}^{t}(v)=\Pr(W_{u}(t)=v)$ be the $t$--step transition probability. We assume $G$ is connected and aperiodic, so that
random walk $W_{u}$ on $G$  has  stationary distribution
$\pi$, where $\pi_v=d(v)/(2m)$.

For periodic graphs, we can replace the simple random walk by a {\em lazy} walk,
in which at each step there is a 1/2 probability of staying put.
By ignoring the steps when the particle does not move in the lazy walk
we obtain the underlying simple random walk. For large $t$, asymptotically half of the steps in the lazy walk will not result in a change of vertex. Therefore  w.h.p. properties of the simple walk after approximately $t$ steps can be obtained from properties of the lazy walk after $2t$ steps.
Making the walk lazy doubles the expected number of returns to a vertex and thus
changes $R_v$ (see \eqref{Rv}) to approximately $2R_v$. As we only consider the ratio $t/R_v=2t/2R_v$ in our proofs, our results will not alter significantly.

Suppose that the eigenvalues of the transition matrix  $P$ are $1=\l_1>\l_2\geq \cdots\geq \l_n$.
Let $\l=\max\set{|\l_i|:i\geq 2}$. By making the chain lazy if necessary, we can always make $\l_2=\max(|\l_2|, |\l_n|)$.

 Let $\Phi_G$ be the conductance of $G$ i.e.
\beq{conduck}
\F_G=\min_{S\subseteq V,\p_S\leq 1/2}\frac{\sum_{x\in S}\p_xP(x,\bar{S})}{\p_S},
\eeq
where $P(x,\bar{S})$ is the probability of a transition from $x \in S$ to $\bar S$.
Then,
\begin{align}
&1-\F_G\leq \l_2\leq 1-\frac{\F_G^2}{2}\label{mix0}\\
&|P_{u}^{t}(x)-\pi_x| \leq (\p_x/\p_u)^{1/2}\l^t.\label{mix}
\end{align}
A proof of these can be found for example in Sinclair \cite{Sin} and  Lovasz \cite{Lo}, Theorem 5.1 respectively.

\paragraph{Mixing time of $ G_r$.}
Let $T$ be  such that, for $t\geq T$
\begin{equation}\label{4}
\max_{u,x\in V}|P_{u}^{t}(x)-\pi_x| =\frac{\min_{x \in V}\,\pi_x}{n^3}=\frac{1}{n^4}.
\end{equation}
By assumption \eqref{nice2} (a result of Friedman \cite{Fried})
 we have $\l\leq (2\sqrt{r-1}+\e)/r\leq 29/30$.
In which case  we can take
\beq{Tlogn}
T(G_r)\leq 120\log n.
\eeq
If inequality \eqref{4} holds, we say the distribution of the walk is in {\em  near stationarity}.

\paragraph{Generating function formulation.}
Fix two vertices $u,v$ of $G$.
 Let $h_t=P_u^t(v)$ be
the probability that  the walk  $W_u$ visits  $v$ at step $t$.
Let
\begin{equation}\label{Hz}
H(z)=\sum_{t=T}^\infty h_tz^t
\end{equation}
generate $h_t$ for $t\geq T$.

We next consider the special case of returns to vertex $v$ made by a  walk $W_v$, starting at $v$.
Let
$r_t=P_v^t(v)$ be the probability  that the  walk returns to
$v$ at step $t = 0,1,...$. In particular note that $r_0=1$, as the walk starts at $v$.
Let
$$R(z)=\sum_{t=0}^\infty r_tz^t$$
generate $r_t$,
and
 let
\begin{equation}
\label{Qs} R_T(z)=\sum_{j=0}^{T-1} r_jz^j.
 \end{equation}
 Thus, evaluating $R_T(z)$ at $z=1$, we have $R_T(1) \ge r_0=1$.
Let
\beq{Rv}
R_v=R_T(1)=\sum_{i=0}^{T-1} r_i.
\eeq
The quantity $R_v$, the expected number of returns to $v$
during the mixing time,  has a particular importance in our proofs.

For $t\geq T$ let $f_t=f_t(u \rat v)$ be  the probability that the
first visit made to $v$ by the walk $W_u$ to $v$ {\em in the period}
$[T,T+1,\ldots]$ occurs at step $t$.
 Let
$$F(z)=\sum_{t=T}^\infty f_tz^t$$
generate  $f_t$.
The relationship between $h_j$ and $f_j,r_j$ is given by
\beq{hfr}
h_t = \sum_{k=1}^t f_k r_{t-k}.
\eeq
In terms of generating functions, this becomes
\begin{equation}
\label{gfw} H(z)=F(z)R(z).
\end{equation}
The following lemma gives the probability that a walk, starting from near stationarity
makes a first visit to vertex $v$ at a given step. The content of the lemma is to extend $F(z)=H(z)/R(z)$ analytically beyond $|z|=1$ and extract the asymptotic coefficients.
For the proof of Lemma \ref{First} and Corollary \ref{geom}, see Lemma 6 and Corollary 7 of \cite{CFGiant}.
We use the lemma to estimate  $\E |\cR_T(t)|$, the expected number of vertices unvisited after $T$. The value of $\E |\cR_T(t)|$
 differs from $\E |\cR(t)|$ by at most $T$ vertices, so as $T=O(\log n)$ and
 $\E |\cR_T(t)|=\Th(n)$ this simplification will not affect our results.

\begin{lemma}
\label{First}
For some sufficiently large constant $K$, let
\begin{equation}
\label{lamby} \l=\frac{1}{KT},
\end{equation}
where $T$ satisfies \eqref{4}.
Suppose that
\begin{description}
\item[(i)]
For some constant $\th>0$, we have
$$\min_{|z|\leq 1+\l}|R_T(z)|\geq \th.$$
\item[(ii)]$T\pi_v=o(1)$ and $T\pi_v=\Omega(n^{-2})$.
\end{description}
There exists
\begin{equation}\label{pv}
p_v=\frac{\pi_v}{R_v(1+O(T\pi_v))},
\end{equation}
 such that for all $t\geq T$,
\begin{align}
f_t(u \rat v)&=(1+O(T\pi_v))\frac{p_v}{(1+p_v)^{t+1}}+O(T \pi_v e^{-\l t/2}).\label{frat}\\
&=(1+O(T\pi_v))\frac{p_v}{(1+p_v)^{t}}\qquad for\ t\geq \log^3n. \label{frat1}
\end{align}
\end{lemma}
Lemma \ref{First} depends on two conditions  (i), (ii). For nice $G_r$, as as $T \pi_v=O(\log n/n)=o(1)$, condition (ii) holds. For the case where $R_v \ge 1 $ constant, it was shown in \cite{CFRHyper} Lemma 18 that condition (i) always holds. The following corollary follows directly by adding up $f_s(u \rat v)$ for $s\ge t$.

\begin{corollary}
\label{geom}
For $t\geq T$ let $\cA_v(t)$ be the event that $W_u$ does not visit $v$ at steps $T,T+1,\ldots,t$.
Then, under the assumptions of Lemma \ref{First},
\begin{align}
\Pr_W(\cA_v(t))&=\frac{(1+O(T\p_v))}{(1+p_v)^{t}} +O(T^2\p_ve^{-\l t/2})\label{atv}\\
&=\frac{(1+O(T\pi_v))}{(1+p_v)^t}\qquad for\ t\geq \log^3n. \label{atv1}
\end{align}
We use the notation $\Pr_W$ here to emphasize that we are
dealing with the probability space of walks on a fixed $G$.
\end{corollary}

Corollary \ref{geom} gives the probability of not visiting a single vertex in time $[T,t]$. We
need to extend this  result to certain small sets of vertices.
In particular we  need to consider sets
consisting of $v$ and a subset of its neighbours $N(v)$. Let $S$ be such a subset.

Suppose now that $S$ is a subset of $V$ with $|S|=o(n)$.
By contracting $S$ to single vertex
$\g=\g(S)$, we form a graph $H=H(S)$ in which the set $S$ is replaced by
$\g$ and the edges that were contained in $S$ are contracted to loops.
The probability of no visit to $S$ in $G$ can be found (up to a multiplicative error of
$1+O(1/n^3)$) from
the probability of a first visit to $\g$ in $H$. This is the content of Lemma \ref{contract} below.

We can estimate the mixing time of a  random walk on $H$ as from the conductance of $G$ as follows.
Note  that the conductance of $H$ is at least that of $G$.
As some subsets of vertices of $V$ have been removed by the contraction of $S$,  the set of
values that we minimise  over, to calculate the conductance of  $H$, (see \eqref{conduck}),
 is a subset of the set of values that we minimise over for $G$.
It follows that the {conductance of $H$ is bounded below by the conductance of $G$.
Assuming that the conductance of $G$ is constant, which is the case in this paper,
then using \eqref{mix0}, \eqref{mix},
we see that the } mixing time for $W$ in $H$ is  $O(\log n)$.

Say that the stationary distribution $\pi_G$ of the walk in $G$ and  $\pi_H$ of the walk in $H$ are {\em compatible} if $\pi_H (\g(S))=\sum_{v \in S}\pi_G(v)$
and for $w \not \in S$, $\pi_G(w)=\pi_H(w)$.
For example, if $G$ is an undirected graph then
the stationary distributions are always compatible, because
the stationary distribution of $\g(S)$ is given by $\pi_H(\g(S))=d(S)/2m=\sum_{v \in S} \pi_G(v)$.
If $G$ is directed, compatibility does not follow automatically, and needs to be checked.

\begin{lemma}\label{contract}\cite{CFGiant}
Let $W_u$ be a random walk in $G$ starting at $u \not \in S$,
and let $\cX_u$ be a random walk in
$H$ starting at $u \ne \g$.
Let $T$ be a mixing time satisfying \eqref{4} in both $G$ and $H$. Then
provided $\pi_G$ and $\pi_H$ are compatible,
\[
\Pr(\cA_{\g}(t);H)=\Pr(\wedge_{v \in S} \cA_{v}(t);G)\brac{1 +O\bfrac{1}{n^3}},
\]
where the probabilities are those derived from the walk in the given graph.
\end{lemma}
\proofstart
Let $W_x(j)$ (resp. $X_x(j)$)  be the position of walk $W_x$ (resp. $\cX_x(j)$)
 at step $j$.
Let $\G=G,H$ and let  $P_u^s(x;\G)$ be the
transition probability in $\G$, for the walk to go from $u$ to $x$ in $s$ steps.
\begin{eqnarray}
\Pr(\cA_{\g}(t);H)&=&\sum_{x\ne \g}P^{T}_{u}(x;H)\;
\Pr(X_x(s-T)\neq \g,\; T\leq s\leq t; H)
\nonumber\\
&=&\sum_{x\ne \g}\brac{\pi_H(x)(1+O(n^{-3}))}
\Pr(X_x(s-T)\neq \g,\; T\leq s\leq t; H)\label{close}\\
&=&\sum_{x\ne \g}\brac{\pi_G(x)(1+O(n^{-3}))}
\Pr(X_x(s-T)\neq \g,\; T\leq s\leq t; H)\label{cclose}\\
&=& \sum_{x \not \in S}
\brac{P^{T}_{u}(x;G)(1+O(n^{-3}))} \Pr(W_x(s-T)\not \in  S,\;T\leq s\leq t;G)\label{extra1}\\
&=& \Pr(\wedge_{v \in S} \cA_{v}(t);G)(1 +O(1/n^3))\nonumber.
\end{eqnarray}
Equation \eqref{close} follows from \eqref{4}. Equation \eqref{cclose} from compatibility of $\pi_H$ and $\pi_G$.
Equation \eqref{extra1} follows because there is a natural measure preserving
map $\f$ between walks in $G$ that start at $x\not \in S$ and avoid $S$ and walks
in $H$ that start at $x \ne \g$ and  avoid $\g$.
\proofend

\section{Simple random walk. Proof of Theorem \ref{SRW}.}
\label{proofs}

\subsection{Degree sequence of the vacant net}\label{degseq}

We need some definitions.
For any edge $e$ of $G$, we say $e$ is {\em red at $t$} if the walk made no transition  along $e$ during $[T,t]$.
If $e$ is a red edge,
we  say $e$ is {\em unvisited} at $t$, (i.e. unvisited between $T$ and $t$).
For any vertex $v$, we assume there is a labeling
 $e_1(v),...,e_r(v)$ of the edges incident with vertex $v$.
 Sometimes we   write $e(v)$
 for a particular edge incident with $v$.
If $v$ has exactly $s$ red edges at $t$,
we say the {\em red degree} of $v$ is $s$, and write $d_R(v,t)=s$.
Recall  that if a vertex $v$ is {\em nice} ($v \in \cN$),  then it is tree-like to
depth least $\ell_1=\e_1\log_rn$.

\begin{lemma}\label{lemsip}
For $\ell=1,...,r$, let
\beq{as}
\a_\ell= \frac{\ell}{r} \brac{2-\brac{\frac{1}{r-1}+ \frac{\ell (r-1)}{r(r-2)+\ell}}}.
\eeq
For $u \in \cN$, let $e_1,...,e_\ell$ be a set of edges incident with $u$. Let
\beq{Pell}
\ol P(u,\ell,t)= \Pr(\text{edges } e_1, \cdots, e_{\ell} \text{ are red at } t ),
\eeq
then
\beq{Qvst}
\ol P(u, \ell,t) =  \exp \brac{-\a_\ell \frac{t}{n} \; \ooi}.
\eeq

\end{lemma}

\begin{proof}
Let $S=\{e_1,...,e_\ell\}$ be a set of edges incident with a nice vertex $u$
of the graph $G$. To prove \eqref{Pell} we need to apply the results of Lemma
\ref{First} and Corollary \ref{geom} to the set $S$. As $S$ is not a vertex
the results of Corollary \ref{geom} do not apply directly, but we can get round this.
We define a graph $H$ with distinguished vertex $\g(\ell)$,
 obtained by modifying the structure of $S$ in $G$ in way detailed below, which we call
 {\em subdivide-contract}.
The graph $H$ is obtained as follows:
\\
(i) Subdivide the edges $e_i=(u,v_i), \; i=1,...,\ell$ incident with vertex $u$
into $(u,w_i),\;(w_i,v_i)$ by inserting a vertex $w_i$.\\
(ii) Contract $\set{w_1,...,w_{\ell}}$ to a vertex $\g(\ell)$, keeping the parallel edges that are created, and let $H$
be the resulting multigraph obtained from $G$ by this process.

We apply Corollary \ref{geom} to $H$ with $v=\g(\ell)$.
 Let $W_x$ be a walk in $H$ starting from vertex $x$.
Let $p_{\g(\ell)} \sim
 \pi_{\g(\ell)}/R_{\g(\ell)}$
 as given in \eqref{pv}. Here $\pi_{\g(\ell)}$ is the stationary probability of
 $\g(\ell)$ and $R_{\g(\ell)}$ is given by \eqref{Rv}.
For $t\geq \log^3n$ let $\cA_{\g(\ell)}(t)$ be
the event that $W_x$ does not visit $\g(\ell)$ at steps $T,T+1,\ldots,t$. Then
from \eqref{atv1}
\beq{PAW}
\Pr_W(\cA_{\g(\ell)}(t))
=\frac{(1+O(T\pi_\g))}{(1+p_{\g(\ell)})^t}.
\eeq
We next prove that
\beq{pig}
p_{\g(\ell)}  =\ooi \frac{\a_\ell}{n},
\eeq
where $\a_\ell$ is given by \eqref{as}.
The first step is to obtain $\pi_{\g(\ell)}$ and $R_{\g(\ell)}$.
 By direct calculation
 \beq{piee}
\pi_\g(\ell) = \frac{2 \ell}{rn+2\ell}.
\eeq
We next prove that $R_{\g(\ell)}=\ooi 1/ (1-f_\g)$, where $\g=\g(\ell)$, and
\beq{eff}
f_\g= { \frac12}\brac{\frac{1}{r-1}+\frac{\ell(r-1)}{r(r-2)+\ell}}.
\eeq
Before we inserted $w_1,...,w_\ell$ into $S$ and contracted them to $\g$, the vertex $u$ was tree-like to depth $\ell_1$.
Let $D(u)=D_{\ell_1}(u)$ be the subgraph of $G$ induced by the vertices at distance at most $\ell_1$ from $u$. Let $\cT_u$ be an infinite $r$-regular tree rooted at $u$.
Thus $D(u)$ can be regarded as the subgraph of $\cT_u$ induced by the vertices
at distance at most $\ell_1$ from $u$. In this way we extend $D(u)$ to an infinite $r$-regular tree.
Let $D'$ be the corresponding subgraph in $H$, and let $\cT'_u$ be the
corresponding infinite graph. Apart from $\g(\ell)$ which has degree $2 \ell$ and
$\ell$ parallel edges between $\g(\ell)$ and $u$, the graph $\cT'_u$ has the same $r$-regular structure as $\cT_u$.

\ignore{
In $H$, vertex $u$ has $\ell$ edges to $\g$, and $r-\ell$ edges to
distinct vertices $x_1,...,x_{r-\ell}$  other than $\g$.
Each of the vertices $z\in \{x_i,v_j,i=1,...,r-\ell, j=1,...\ell\}$
has $r-1$ { edges} other than the one incident with $\g$ or $u$.
In $\cT'_u$, each of these $r-1$  edges of $z$ extends
   to an infinite $r$-regular tree ${\cT_{z}}$ which is a sub-tree of $\cT'_u$.
}

Let $\cal T$ be an infinite $r$-regular tree rooted at a fixed vertex $v$ of arbitrary positive degree $d(v)$.  Lemma \ref{lemRv} proves that the probability $\f$ of a first return to $v$ in $\cT$  is given by $\f=1/(r-1)$.
Let $f_\g$ be the probability of   a first return to $\g$ in $\cT'_u$.
With probability $1/2$ a walk starting at $\g$  passes to one of $v_1,...,v_\ell$ in which case the probability of a return to $\g$ is $\f=1/(r-1)$.
With probability $1/2$ the walk
 passes from $\g$ to $u$ from whence it returns to $\g$ with probability $\ell/r$ at each visit to $u$. If the walk exits to a neighbour of $u$ other than $\g$
the probability of a return to $u$ is $\f=1/(r-1)$.
Thus in $\cT'_u$, a first return to $\g$ has probability
 \begin{align*}
 f_\g= & { \frac12}\brac{ \f + \frac{\ell}{r} \sum_{k \ge 0} \brac{\bfrac{r-\ell}{r}\f}^k}\\
 =&{ \frac12}\brac{ \f + \frac{\ell
 }{r-(r-\ell)\f}}.
 \end{align*}
This establishes \eqref{eff}.
It  follows from  Lemma \ref{lemRv} that
the value of $R_{\g(\ell)}=\ooi 1/ (1-f_\g)$.
 Combining \eqref{piee} and \eqref{eff} gives the value of $p_{\g(\ell)}$ in \eqref{pig} where $\a_\ell$ is \eqref{as}.

The last step is to get back from the walk in $H$ to the walk in $G$.
 By Lemma \ref{contract}, the event that $\g(\ell)$ is unvisited  at
 steps $T,...,t$ of a random walk in $H$, has the same asymptotic
probability as the event \eqref{Pell} in $G$ that there is no transition along the edge set $\{e_i=(u,v_i), \; i=1,...,\ell\}$
during steps $T,...,t$ of a random walk in $G$.
This, and   \eqref{PAW} gives
\[
\ol P(u,\ell, t)= \ooi \Pr_W(\cA_{\g(\ell)}(t))
=\frac{(1+o(1))}{(1+p_{\g(\ell)})^t}=\ooi e^{-tp_{\g(\ell)}(1+O(p_{\g(\ell)}))}.
\]
This,  along with \eqref{pig}
  completes the proof of the lemma.
\end{proof}

Let $d_R(v,t)$ be the red degree of vertex $v$ at step $t$ and let
 $S(v,s,t)={d_R(v,t) \choose s}$ be
 the number of $s$-subsets of red edges incident with vertex $v$ at step $t$.
Let $M(s,t)$ be given by
\[
M(s,t)= \sum_{v \in \cN} S(v,s,t).
\]
Thus $M(s,t)$ enumerates sets of incident red edges of size $s$ over nice vertices.

Recall that we have defined an edge to be red if it is unvisited in $T,...,t$.
By definition,  all edges  start red at step $T$.
For $t \ge T$,
the random variable $M(s,t)$ is monotone non-increasing in $t$. For any $s \ge 1$ there
will be some step $t(s)$ at which $M(s,t(s))=0$.

\begin{lemma}\label{th4}
Let $\a_s$ be given by \eqref{as}. The following results hold w.h.p.,
\begin{enumerate}[(i)]
\item
\beq{Mst}
\E M(s,t)= \ooi n {r \choose s} \exp \brac{- p_{ \g(\ell)}}
=
\ooi n {r \choose s} \exp \brac{-\ooi\a_s \frac{t}{n}}.
\eeq
\item
For $s \ge 1$ let $t_s= ( n \log n)/\a_s$.
The values $t_s$ satisfy $t_r < t_{r-1} < \cdots < t_1$.\\
Let $\om=\e \log n$. For $t<t_s-\om n$, $\E M(s,t) \rai $
whereas for $t>t_s+\om n$, $\E M(s,t)=o(1)$. \\
For { $t=O(n)$}, $|\ol \cN| =o(\E M(s,t))$.
\item
For all $0 \le t \le t_s-\om n$,   the value of $M(s,t)$ is concentrated within
$ \ooi \E M(s,t)$.
\end{enumerate}
\end{lemma}

\begin{proof}
{\em (i), (ii).}
The value of $\E M(s,t)$ follows  from \eqref{Qvst} by linearity of
expectation, and the fact that $|\cN|= (1-o(1)) n$. Thus
\beq{EMst}
\E M(s,t)= \sum_{u \in \cN} {r \choose s} \ol P(u, s,t) =\ooi n {r \choose s}
e^{-  \a_s \frac tn \sooi}.
\eeq
For $t \le t_s-\om n$, $\E M(s,t) =\Om (n^{\e})$.

The function $\a_s$ is { strictly} monotone increasing in $s$.
For $r \ge 3$, the derivative $d \a(x)/dx$ is positive for $x \in [0,r)$,
and zero at $x=r$.
Thus the values $t_s$ satisfy $t_i<t_j$ if $i >j$.

{\em Proof of (iii).}
Fix $s,t$ where $s =1,...,r$, and $t \le t_s-\om n$.
We use the Chebyshev inequality to prove concentration of $Z=M(s,t)$.
Suppose that $\d \ll \e$, and $\om'(n)=\d \log n$, then
\beq{omdash}
\log\log n\ll \om'=\om'(n)= \d \log n \ll \om =\e \log n.
\eeq
We first show that
\beq{cheby}
\Var(Z)=\E Z+O(r^{\om'}\E Z)+e^{-a\om'}(\E Z)^2,
\eeq
for some constant $a>0$.

Let $v, w \in\cN$. Let $Q_s(v)=\{e_1(v),...,e_s(v)\}$ be a set of edges incident with $v$,
and let $Q_s(w)=\{f_1(w),...,f_s(w)\}$ be a set of edges incident with $w$. Let $\cE_v=\cE(Q_s(v))$
be the event that the edges in $Q_s(v)$ are red at $t$.
Similarly, let  $\cE_w=\cE(Q_s(w))$ be the event
that  the $Q_s(w)$ edges are red at $t$.

Let $v,w$ be at distance at least
$\om'$ apart then we claim that
\beq{var}
\Pr(\cE_v\cap \cE_{w})=(1+e^{-\Omega(\om')})\;\Pr(\cE_v)\Pr(\cE_{w}).
\eeq
To prove this we use the  same method as  Lemma \ref{lemsip}. That is to say,
we use Corollary \ref{geom} to find the unvisit probability of a vertex
$\g$ that we construct from $Q_s(v) \cup Q_s(w)$ using subdivide-contract.
We carry out the subdivide-contract
process on the edges of $Q_s(v), Q_s(w)$ by inserting an extra vertex $x_i$ into $e_i$
and an extra vertex $y_i$ into $f_i$, and contracting $S=\{x_1,...,x_s,y_1,...,y_s\}$ to $\g(S)$.

For the
random walk on the associated graph $H=H(\g(S))$ we have that
$p_{\g(S)}$ in \eqref{pv} is given by $p_{\g(S)} \sim \pi_{\g(S)}/R_{\g(S)}$,
where
\[
\pi_{\g(S)}= \frac{4s }{rn+4s} .
\]
By Lemma \ref{Rv} we can write $1/R_{\g(S)}=\ooi(1-f_{\g(S)})$.
We next prove that the value of $f_{\g(S)}$ is given by
\[
f_{\g(S)}=\frac{1}{2}\brac{f_{\g(S_x)} +f_{\g(S_y)}+O(f^*)}.
\]
In this expression, $f^*$ is an error term defined below, and $\g(S_x), \g(S_y)$ are the contractions of
$S_x=\{x_1,...,x_s\}$, and  $S_y=\{y_1,...,y_{s}\}$ respectively, as obtained
in Lemma \ref{lemsip} and (e.g.) $f_{\g(S_x)}$
is evaluated in $H(\g(S_x)$.
Indeed, with probability $1/2$, the first move from $\g(S)$ will be to a vertex $u$
which is a neighbour of one of $S_x=\{x_1,...,x_s\}$ on the the subdivided edges $e_1,..,e_s$.
Assume it is to a neighbour of $S_x$. The probability of a first return directly to $\g(S_x)$ will be
 $f_{\g(S_x)}=\ooi f$  as given by Lemma \ref{lemsip}.

The $O(f^*)$ term is a correction for the probability that a walk staring from $\g(S_x)$
makes a transition across any of the edges in $Q_s(w)$ during the mixing time.
This event is not counted as a return in walks on $H(\g(S_x))$ but would be in $H(\g(S))$.
However, because $v$ and $w$ are at distance at least $\om'$, using \eqref{feq},
 the probability $f^*$ of a visit to $Q_s(w)$ during $T$
can be bounded by $T(n^{-1}+\l_{\max}^{\om'})$.
Thus
\beq{32}
p_{\g(S)}= (1+O(1/n)+O(Te^{-\Om(\om')})\;\; (p_{\g(S_x)}+p_{\g(S_y)}).
\eeq
Equation \eqref{var} follows on using equation \eqref{32}, Corollary \ref{geom}
with $p_{\g(S)}, p_{\g(S_x)}$ and $ p_{\g(S_y)}$ followed by
Lemma  \ref{contract}.
This confirms \eqref{var} and gives
\[
\Pr(\cE_v { \cap} \cE_w)=(1+e^{-\Om(\om')})\Pr(\cE_v) \Pr( \cE_w)=\ooi  \ol P(v,s,t) \ol P(w, s,t),
\]
where $\ol P(v,s,t)$ is given by \eqref{Qvst} in Lemma \ref{lemsip}.

Summing over $v,w \in \cN$ and  edge sets $Q_s(v), Q_s(w)$ incident with $v,w$ respectively,
\begin{align*}
\E(Z^2(t))&=\E Z+\sum_{\substack{v,w \\ Q_s(v), \;
Q_s(w)\\dist(v,w)\geq \om'}}\Pr(\cE_v\cap \cE_w)+
\sum_{\substack{v,w \\ Q_s(v),\; Q_s(w)\\dist(v,w)< \om'}}\Pr(\cE_v\cap \cE_w)\\
&\leq \E Z+(1+e^{-a\om'})(\E Z)^2+r^{\om'}\E(Z)
\end{align*}
and \eqref{cheby} follows.
Applying the Chebyshev inequality we see that
\beq{toterr}
\Pr \brac{|Z-\E Z|\geq \E Z\;e^{-a\om'/3}}
\leq \frac{2r^{\om'} e^{a\om'}}{\E Z}+e^{-a\om'/3}.
\eeq
When $t \le t_s-\om n$, $\E Z \ge e^{\om \a_s}/2=\Om(n^{\e})\gg n^{\d}$ and our choice of $\om'$ in \eqref{omdash}
implies that we can find a $\d_1$ such that the RHS of \eqref{toterr} is $O(n^{-\d_1})=o(1)$ for such $t$.

The result \eqref{toterr} from the Chebychev inequality
is too weak to prove concentration of $M( s, t)$ directly for all of $t_s$ steps.
We copy the approach used in \cite{CFVac}, { Theorem 4(a)}. Interpolate the interval
$[0,t_s]$ at $A=n^{\d_1/2}$ integer points $s_1,...,s_A$
at distance $\s=t_s\;n^{-\d_1/2}$ apart (ignoring rounding), for
  some small constant $\d_1>0$ determined by \eqref{toterr}.
The concentration at the interpolation points follows from \eqref{toterr}.
We  use the monotone non-increasing property of $M( s, t)$
to bound the value of $M(s,t)$  between $s_i$ and $s_{i+1}$.
The proof of this is identical to
the one in \cite{CFVac} and is not given in further detail here.
 \end{proof}

\subsection{Uniformity: Using random walks in the configuration model}\label{uniformity}
We
use the random walk to generate the
graph $G$ in question.
The main idea is to realize that as $G$ is a random graph,
 the graph $\G(t)$ of the vacant set or vacant net has a simple description.
Intuitively, if we condition on $\cR(t)$ and the  history
of the process, (the walk trajectory up to step $t$),
 and if $G_1,G_2$ are graphs with
vertex set $\cR(t)$ and the same degree sequence, then substituting
$G_2$ for $G_1$ will not conflict with the history.
Every extension of $G_1$ is an extension of $G_2$ and vice-versa.

\ignore{
We will describe the history of the walk in  a way that does not require us to reveal the entire graph.
Suppose we are given $Y \in [r]^t$ and, for any vertex $v$, a function $f_v$
which maps $[r]$ onto the edges incident with vertex $v$. Given an $r$--regular graph $G$ and a start vertex $u$, the sequence $Y=(Y_0,...,Y_{t-1})$ describes a walk trajectory $W_u(t)$ in $G$ as follows.
From $u$, take the edge given by $f_u(Y_0)$, and so on.

Our method of sampling u.a.r. from $\cG_r$, is to generate a random graph in the configuration model using a random walk. Because the configuration model
 assigns labeled points (edge end points, or half edges) to any
vertex $v$, this gives a natural mapping $f_v$ from $[r]$ to half edges.
We randomly pair as many half edges as the sequence $Y \in [r]^t$ needs in order to determine the walk trajectory up to step $t$.
The remaining subgraph generated by the unmatched configuration points is random.
}

We briefly and informally explain what we do.
By working in the configuration model, we can use the
random walk to generate a random  $r$--regular multigraph.
Because the configuration points (half edges) at any vertex have labels,
we can sample u.a.r. from these points to determine the next edge
transition of the walk without exposing all the edges at the vertex in the underlying
multigraph.
In this way the walk discovers the edges of the multigraph as it proceeds.
If we stop the walk at some step $t$,
the undiscovered part of the multigraph is random,
conditional on the subgraph exposed by the walk so far, and
the constraint that all vertices have degree $r$.

We use the configuration or pairing model of Bollob\'as \cite{B}, derived from a counting formula of Canfield \cite{BC}. We start with $n$ disjoint sets
of  $S_1,S_2,\ldots,S_n$ each of size $r$.
The elements of $S_v=\{v(1),...,v(r)\}$ correspond to the labeled endpoints of the half edges incident with
vertex $v$. We refer to these elements as (configuration) points.

Let $S=\bigcup_{i=1}^n S_i$.
A {\em configuration} or {\em pairing} $F$ is a partition
of $S$ into $rn/2$ pairs. Let $\Omega$ be the set of configurations.
Any $F\in\Omega$ defines
an $r$-regular multi-graph $G_F=([n],E_F)$ where $E_F=\set{(i,j):\exists
\set{x,y}\in F:x\in S_i,y\in S_j}$, i.e. we contract
$S_i$ to a vertex $i$ for $i\in [n]$.

Let $U_0=S$, $F_0=\es$. Given $U_{i-1},F_{i-1}$ we  construct $F_i$ as follows. Choose $x_i$ arbitrarily from $U_{i-1}$. Choose $y_i$ u.a.r. from $U_{i-1}\sm \{x_i\}$. Set $F_i=F_{i-1} \cup \{\{x_i,y_i\}\}, U_{i}=U_{i-1} \sm \{x_i,y_i\}$.
If we stop at step $i$, the points in $U_i$ are unpaired, and can be paired u.a.r.
The underlying
multigraph of this pairing of $U_i$ is a random multigraph in which the degree of vertex $v$ is $d(v)=|S_v \cap U_i|$.

It is known that: (i) Each simple graph arises the same number of times
as $G_F$. i.e. if $G,\;G'$  are simple, { then $|\{F:G_F=G\}|=|\{F':G_F'=G'\}|$.}
(ii) Provided  $r$ is constant, the probability
$G_F$ is simple is bounded below by a constant. Thus if $F$ is chosen
uniformly at random from $\Omega$ then any event that
occurs w.h.p. for { $F$, occurs w.h.p. for $G_F$,
and hence w.h.p. for $G_r$.}

We next explain how to use a random walk on $[n]$ to generate a random $F$,
and hence a random multigraph $G$. To do this,
we begin with
a starting vertex $u=i_0$. Suppose that at the $t$--th step
we are at some vertex $i_t$, and  have
a partition of $S$  into red and blue points, $R_t,B_t$ respectively.
Initially, $R_0=S$ and $B_0=\emptyset$.
In addition we have a collection
$F_t$ of disjoint pairs from $S$ where $F_0=\emptyset$.

At step $t+1$ we choose a random edge incident
with $i_t$.
Obviously $i_t \in \cB(t)$, as it is visited by the walk,
but we treat the configuration points in $S_{i_t}$
as blue or red, depending on whether the corresponding
edge is previously traversed (blue) or not (red).
Let $x$ be chosen randomly from $S_{i_t}$.
\ignore{
Recall that the points of $S_{i_t}$
have labels $i_t(\ell),\ell=1,...,r$, so that if $x=i_t(\ell)$ then $Y_t=\ell$.
We refer here to the construction of $Y=(Y_0,...,Y_t)$ in Lemma \ref{lem2},
and we index the edge transition $(i_t,i_{t+1})$ by the point label $\ell$ used in choosing it.
}
There are two cases of how $i_{t+1}$ is chosen.

If $x\in B_t$ then it was previously  paired with a $y\in S_j\cap B_t$,
and thus $j \in \cB(t)$.
The walk moves from $i_t$
to $i_{t+1}=j$ along an existing edge corresponding to some $\{x,y\} \in F$. We let $R_{t+1}=R_t,B_{t+1}=B_t$ and we let $F_{t+1}=F_t$.

If $x\in R_t$, then the edge is unvisited, so we choose $y$ randomly from
$R_t\setminus \set{x}$.  Suppose that $y\in S_j$.
This is equivalent to moving from
$i_t\in \cB(t)$ to $i_{t+1}=j$.
We now check vertex $j$ to see if it was previously visited.
 If $j \in \cB(t)$ this is equivalent to moving
between blue vertices on a previously unvisited edge. If $j \in \cR(t)$, this is equivalent to
moving to a previously unvisited vertex.
In either case we update as follows.
$R_{t+1}=R_t\setminus \set{x,y}$ and $B_{t+1}=B_t\cup\set{x,y}$,
and  $F_{t+1}=F_t\cup\set{\set{x,y}}$.

\ignore{
After $t$ steps we  have constructed the sequence $Y=(Y_0,...,Y_{t-1})$
 which along with $i_0$ gives the history of the process.
}
After $t$ steps we have a random pairing $F_t$ of at most
$t$ disjoint pairs from $S$.
The entries in $F_t$ consist of a known pairing of $B_t$, and constitute the
 revealed edges of the random graph. The
points in $R_t$ are still unpaired.
In principle
we can extend $F_t$ to a random configuration $F$ by adding a
random pairing of $R_t$ to it.  The vacant net, $\wh \G(t)$
is the  subgraph of $V$ induced by the edges  unvisited during steps $1,...,t$,
and is the underlying
multigraph of a u.a.r. pairing of $R_t$.
To generate $\G(t)$, the subgraph induced by
the vacant set $\cR(t)$,  we extend the pairing $F_t$ to a pairing
$F_{t'}$ by method {\bf Extend--$\cB(t)$} defined as follows.

{\bf Extend-$\cB(t)$.}
Let $S_B=\cup_{v \in \cB(t)} S_v$. Let $K=S_B \cap R_{\t}$.
For $\t \ge t$, and while $K \ne \es$
choose an arbitrary  point $x$ of $K$.
Pair $x$
with a u.a.r. point $y$ of
$R_\t-\{x\}$. Let $R_{\t}=R_\t \sm \{x,y\}$. If $y \in K$ let $K=K\sm \{x,y\}$
else let $K=K\sm \{x\}$. Set $\t=\t+1$.
Let $t'=\t$ be the first step at which $K=\es$.
Pair $R_{t'}$ u.a.r. to generate the multigraph $\G(t)$.

The next lemma summarizes this discussion.
\begin{lemma}\label{walkconfig}\
\vspace{-.3in}
\begin{enumerate}[i)]
\item The pairing $F_t$ can be generated in the configuration model by  a random walk $W_u(t)$ without exposing any pairings not in $F_t$.
    The underlying multigraph of $F_t$ gives the edges covered by the walk $W_u(t)$.
\item The pairing $F_t$ plus a u.a.r. pairing of $R_t$ is a {uniform} random member of $\Omega$.
\item The vertex $v \in V$ is in $ \cR(t)$ if and only if  $S_v\subseteq R_t$.
\item {\bf Vacant net.}
The u.a.r. pairing of $R_t$ gives the vacant net, $\wh \G(t)$ as a random multigraph with degree sequence determined by $\wh d(v)=|S_v \cap R_t|$ for $v \in V$.
Let $\wh \bd(t)$ be the degree sequence of $\wh \G(t)$.
Conditional on $\wh \G(t)$ being simple,  $\wh \G(t)$ is a u.a.r. graph with  degree sequence $\wh \bd(t)$.
\item {\bf Vacant set.}
Extend $F_t$ to $F_{t'}$ using method {\bf Extend--$\cB(t)$} described above.
The u.a.r. pairing of $R_{t'}$ gives  $\G(t)$, the induced subgraph of the vacant set, as a random multigraph with degree sequence determined by $ d(v)=|S_v \cap R_{t'}|$ for $v \in \cR(t)$. Let $\bd(t)$ be the degree sequence of $ \G(t)$.
Conditional on $\G(t)$ being simple,  $\G(t)$ is a u.a.r. graph with  degree sequence $\bd(t)$.
\end{enumerate}
\end{lemma}

\subsection{Applying the Molloy-Reed Condition}\label{MRC}

The Molloy-Reed condition for bounded degree graphs can be  stated as follows.
\begin{theorem}\label{MR}
Let $G_{N,\bd}$ be
the graphs with
vertex set $[N]$ and degree sequence
$\bd=(d_1,d_2,\ldots,d_N)$, and endowed with the uniform measure.
Let $D(s)=|\set{j:d_j=s}|$, be the number of vertices of degree $s=0,1,\ldots,r$,
where $D(s)=(1+o(1))\l_sN$ for $s=0,1,\ldots,r$, and
 $\l_0,\l_1,\ldots,\l_r\in [0,1]$ are such that $\l_0+\l_1+\cdots+\l_r=1$.
Let
\beq{MRL}
L(\bd)=\sum_{s=0}^rs(s-2)\l_s.
\eeq
\begin{description}\label{Leq0}
\item[(a)] If $L(\bd)<0$ then w.h.p. $G_{n,\bd}$ is sub-critical.
\item[(b)] If $L(\bd)>0$ then w.h.p. $G_{n,\bd}$ is super-critical.
\end{description}
\end{theorem}

The following theorem on the scaling window is adapted from Theorem 1.1 of Hatami and Molloy
\cite{HM}, with the observation (after Theorem 3.2) from
\v{C}erny and  Teixeira \cite{CT2} that including a constant proportion of  vertices of degree zero does not modify the validity of the result.
\begin{theorem}\label{HM}
\cite{HM}
Let $G_{N,\bd}$ be
the graphs with
vertex set $[N]$ and degree sequence
$\bd=(d_1,d_2,\ldots,d_N)$, and endowed with the uniform measure.
Let $R=\sum_{u \in V} d_u(d_u-2)^2/2|E(G)|$. Assume that  $R>0$ constant, and $D(2)< N(1-\d)$ for some $\d>0$. For any $c>0, \e>0$, and
$-cN^{2/3} \le N L(\bd) \le c N^{2/3}$,
\[
\Pr(|C_1| = \Th(N^{2/3}))\ge 1-\e.
\]
\end{theorem}

To complete the proof of Theorem \ref{SRW} we need to evaluate $L(\bd)$ for $\wh \G(t)$ to obtain $\wh t$.
It is convenient for us to express $L(\bd)=\sum_{s=0}^rs(s-2)\l_s$ in a form which uses the results of Lemma \ref{lemsip} and Lemma \ref{th4} of Section \ref{degseq}.

\begin{lemma}\label{MMR}
Let $G=(V, E, \bd)$ be a graph with degree sequence $\bd$ of maximum degree $r$.
Let $D(s), s=0,...,r$,
be the number of vertices
of degree $s$. Let $U \seq V$ be a set of vertices, and $\ol U= V \sm U$.
Let $M_U(s)=\sum_{u \in U} {d(u) \choose s}$, and let
$R=\sum_{u \in V} d(u)(d(u)-2)^2/2|E(G)|$.
Then  $L(\ul d)$ can be written as
\beq{newLd}
L(\ul d) \cdot N= \ooi \brac{2M_U(2)-M_U(1)+O( r^2 |\ol U|)},
\eeq
and  $R$ can be written as
\beq{newR}
 R \cdot \brac{M_U(1)+O(r |\ol U|)}= \ooi \brac{ 6 M_U(3)- 2M_U(2)+M_U(1)
  + O(r^3 |\ol U|)}.
\eeq
\end{lemma}

\begin{proof}
Let
\beq{newMR}
Q= \sum_{s=0}^r s(s-2) D(s),
\eeq
then $Q$ can be written as
\begin{eqnarray}\nonumber
Q&=& \sum_{s=0}^r s(s-1)D(s)-\sum_{s=0}^r sD(s)\\
&=& \sum_{v \in V} d(v)( d(v)-1) -\sum_{v \in V} d(v) \nonumber
\\
&=& \sum_{v \in U} d(v)( d(v)-1) -\sum_{v \in U} d(v) +
\brac{\sum_{v \not\in U} d(v)( d(v)-1) -\sum_{v \not\in U} d(v)}\nonumber\\
&=& 2M_U(2)-M_U(1)+O(r^2 |\ol U|).\label{better}\label{QQ}
\end{eqnarray}
The  case for $R$ is similar.
\end{proof}

In our proofs,
we choose $U =\cN$, the set of nice vertices.
It follows from Lemma
\ref{th4} that $r^2\ol U=o(M_U(1)+M_U(2))$.
The next lemma proves the Molloy-Reed threshold condition is
equivalent to  $M_\cN(1) \sim 2 M_\cN(2)$.

\begin{lemma}\label{mylem}
\begin{enumerate}[(i)]
\item
The asymptotic solution to $L(\bd)=0$ in \eqref{MRL} obtained at $\wh t= \ooi \th^*n$ where
\beq{thstar}
\th^*= \frac{r(r^2-2r+2)}{2(r-2)^2} \log (r-1).
\eeq
\item The assumptions of Theorem \ref{HM} are valid and the scaling window is
of order $\Th(n^{2/3})$.
\end{enumerate}
\end{lemma}

\begin{proof}
Let $\bd$ be the degree sequence of $\wh \G(t)$,  let $\ul D$
be the degree sequence of nice vertices $\cN$, and  $\ul{\ol D}$ the degree sequence
of $\ol \cN$.
For nice vertices and any $0 \le s \le r$
we  use the notation $M(s,t)=M_{\cN}(s,t)$.
Thus using \eqref{newLd} with $U=\cN$,
\beq{nbd}
nL(\bd) \sim Q(\bd) =2M(2,t)-M(1,t)+O(r^2|\ol  \cN|)-O(T).
\eeq
Thus the condition { $L(\bd)\sim 0$} is equivalent to $Q(\ul D)/n\ra 0$.
The term $O(T)$ removes any vertices/edges visited during the mixing time $T$,
but unvisited during $T,...,t$ and hence marked red.
From \eqref{nice}, $|\ol \cN|=O(n^{\e})$.
For  nice vertices, and $t=cn$ for any $c \ge 0$ constant
gives $M(2,t)=\Th(n), M(1,t)=\Th(n)$.
Thus when $M(1,t) \sim  2M(2,t)$ then $L(\bd)\sim0$.
By Lemma \ref{th4}, $M(s,t)$ is asymptotic to \eqref{Mst}, which is
\begin{align*}
M(1,t)= & \ooi r \exp \brac{-\frac{t}{n} \frac{2(r-2)}{r(r-1)}}\\
M(2,t) = & \ooi r(r-1)/2
 \exp \brac{-\frac{t}{n} \frac{2}{r} \brac{2-\brac{\frac{1}{r-1}+\frac{2(r-1)}{r(r-2)+2}}}}.
\end{align*}
Thus $ L(\bd) \ra 0$ when $t \sim \wt t = \th^* n$ where $\th^*$ is given by \eqref{thstar}.

Regarding the expression for $R=R(t)$ in \eqref{newR}, with $U=\cN$.
\[
(M(1,t)+O(r \ol \cN)) R(t) = \ooi (6 M(3,t)- 2M(2,t)+M(1,t)) + O(r^3 \ol \cN).
\]
By Lemma \ref{th4}, for $t=cn$, any $c \ge 0$ constant, and $s=1,2,3$
we have that w.h.p. $M(s,t)=\Th(n)$. On the other hand from \eqref{nbd},
and the assumption of the scaling window
\[
\frac{2M(2,t)-M(1,t)) + O(r^3 \ol \cN)}{M(1,t)}= O(L(\bd))=o(1).
\]
Thus $R(t)>0$ constant.
\end{proof}

\section{Non-backtracking random walk. Proof of Theorem \ref{nobac}}
\label{noback}

Note that, as in the case of a simple random walk, we can use
a non-backtracking random walk  to generate the underlying graph in the configuration model.
The only change to the sampling procedure given in Section \ref{uniformity},
is as follows. Suppose the walk arrives at vertex $v$ by a transition $(u,v)$.
In the configuration model, this is equivalent to a pairing $\{x(u), y(v)\}$
where $x(u) \in S_u, \; y(v) \in S_v$. To make the walk non-backtracking,
we sample the configuration point of $v$ used for the next transition u.a.r from
$S_v \sm \{y(v)\}$.

For a connected  graph $G=(V,E)$ of minimum degree 2, the state space of a non-backtracking walk $W$ on $G$ can be described by a digraph $M=(U,D)$ with vertex set $U$ and  directed edges $D$. To avoid any confusion with the vertex set $V$ of $G$, we refer to the elements $\s$ of $U$ as states, rather than vertices.
The states $\s \in U$ are orientations $(u,v)$  of edges $\{u,v\} \in E(G)$. The state $\s=(u,v)$ is read as `the walk $W$ arrived at $v$ by a transition
along $(u,v)$'.  Let $N(u)=N_G(u)$
denote the neighbours of $u$ in $G$.
The in-neighbours of $(u,v)$ in $M$ are states $\{(x,u), x \in N(u), x \ne v\}$.
Hence the state $(u,v)$ has in-degree $(r-1)$ in $M$.
Similarly $(u,v)$ has out-degree $(r-1)$ and out-neighbours $\{(v,w), w \in N(v), w \ne u\}$.

Let $\cM$ be a simple random walk on $M$. The walk $\cM$ on $M$ is a Markov
process which corresponds directly to the non-backtracking walk on $G$.
For  states $\s=(u,v)$, $\s'=(v,w)$, the transition
matrix $P=P(\cM)$ has entries $P(\s,\s')=1/(d(v)-1)$ if $w \ne u$ and $P(\s, \s')=0$
otherwise. The total number of states $|U|=2|E(G)|=2m$.
Using $\pi = \pi P$,
\[
\pi(u,v)= \sum_{x \in N(u), x \ne v}\frac{\pi(x,u)}{d(u)-1},
\]
which has solution $\pi(\s)=1/2m$.

For random $r$-regular  graphs,
 Alon et al.  \cite{Alon} established that a non-backtracking walk  on $G$ has
mixing time $T_G=O(\log n)$ w.h.p. The analysis in \cite{Alon} was made on the graph $G$ whereas, to apply  Corollary \ref{geom}, we need the mixing time $T_{\cM}$ of the Markov chain $\cM$. The proof of Lemma \ref{snots} below is given in Section \ref{snotproof} of the Appendix.
\begin{lemma}\label{snots}
For $G \in G_r$, $r \ge 3$ constant, w.h.p. $T_{\cM}=O(\log n)$.
\end{lemma}

 In Section \ref{proofs} we described a technique
 called subdivide-contract which we used to obtain first returns to a
suitably constructed set $S$ which was contracted to a vertex $\g(S)$.
 It remains to
establish the value of $R_{\g(S)}$ obtained by applying
the subdivide-contract method to
 the various sets $S$ of vertices and edges used in our proof.
In each case we outline  the construction of the set $S$ and state the
relevant value of $p_{\g(S)}$ as given by \eqref{pv} which we use in Corollary \ref{geom}.
Because the walk cannot backtrack, the calculation of  $R_\g$
for sets $S$ of tree-like (i.e. nice) vertices is  greatly simplified.
Let $\cT_{\g}$ be an infinite $r$-regular tree  rooted at a vertex $\g$
of arbitrary degree.
For a non-backtracking walk starting from $\g$, a first return to $\g$ after moving to an adjacent vertex,
is impossible.

\subsection{Properties of the vacant set}

{\bf Size of vacant set.}
Let $v$ be a nice vertex of $G$, and let $S=[v]$ be a set of states of $M$, where $[v]=\{(u,v), u \in N(v)\}$.
A visit to $[v]$ in $M$ is equivalent to a visit to $v$ in $G$.
If $v$ is nice then,
(i) states $(u,v), (x,v) \in [v]$ are directed distance at least $ 2\ell=\e \log_r n$ apart in $M$; (ii) the state $(u,v)$ induces an $(r-1)$-regular in-arborescence and out-arborescence in $M$.

Contract the set $[v]$ of states of $M$ to a single state $\g([v])$ retaining all edges incident with $[v]$.
  This gives a multi-digraph $H$ with states $\wt U=(U\sm [v]) \cup \{\g([v])\}$. We only apply this construction to nice vertices $v$, in which case the digraph rooted at $\g([v])$ is an arborescence to depth $\ell$.
To simplify notation, if we contract a set $S$ of states of $M$ to $\g(S)$,
and $f$ is any state of $M$ not in $S$, we use the indexing $f \not \in S$, both for $M$ and $H$, i.e. as shorthand for $f \not = \g([v])$.

The set $[v]$ consists of $r$ states of $U$ each of in-degree and out-degree $(r-1)$.
 As we contracted without removing edges, the vertex $\g([v])$ has in-degree and out-degree $r(r-1)$. For any state $(v,w)$ of $H$ there are $r-1$ parallel edges directed from $\g([v])$ to $(v,w)$
and no others. For a state $\s$ of $H$, let $N^-(\s)$ be the in-neighbours of $\s$, and let $d^+(\s)$ be the out-degree of $\s$.

Let $\cH$ be a simple random walk on $H$. Apart from transitions to and from $[v]$ (resp. $\g([v])$), the transition matrices of the walks $\cM$ and $\cH$ are identical.
Let $\pi$ be the stationary distribution of $P(\cM)$  in $M$ and
$\wt \pi$  the stationary distribution of $\wt P(\cH)$ in $H$.

For irreducible aperiodic Markov chain with transition matrix $P$,
the stationary distribution is the unique  vector of probabilities $\pi$ which
satisfies the equations $\pi= \pi P$. Given $\pi$ we only have to check this condition.

We claim that $\wt \pi(\g([v]))=1/n$.
For any state $f$ of $H$ other than $\g([v])$, we claim  $\wt \pi (f)=1/rn$,
and thus $\wt \pi(f)=\pi(f)$ for such states. This includes out-neighbours
$(v,w)$ of $\g([v])$.
Considering $\wt \pi=\wt \pi \wt P$, we have
\begin{eqnarray}
\wt \pi(\g([v]))&=& \sum_{f \in N^-(\g([v]))}\frac{\wt \pi(f)}{d^+(f)}
\label{L1}\\
\wt \pi(v,w)&=& \wt \pi(\g)\frac{ (r-1)}{d^+(\g)}=\frac{\wt \pi(\g([v]))}{r}. \label{L2}
\end{eqnarray}
For \eqref{L1}, as $d^-(\g([v]))=r(r-1)$ and $d^+(f)=(r-1)$, this
confirms $\wt \pi(\g)=1/n$.
For \eqref{L2}, the $(r-1)$ comes from the
parallel edges from $\g([v])$ to $(v,w)$,
 and confirms $\wt \pi(v,w)=1/rn$.
For any other state $f$, the relevant rows of $\wt P$ are identical with those of  $P$
confirming $\wt \pi(f)=\pi(f)=1/rn$.

We  use Lemma \ref{contract} to apply results obtained for $\cH$  to the walk $\cM$. The lemma needs  the stationary distributions $\pi=\pi_{\cM}$ and
$\wt \pi=\pi_{\cH}$ to be compatible i.e.  $\pi(f)=\wt \pi(f)$ for $f \not \in [v]$ (resp. $f \not = \g([v])$).
This follows immediately from the values of $\pi, \;\wt \pi$ obtained above.

Finally we calculate $R_{\g([v])}$. We first give a general explanation of the method.
 Let $\cT_\g$ be an infinite  arborescence with root vertex $\g$ of out-degree $r(r-1)$ and all other vertices of out-degree $(r-1)$.
 Similar to Lemma \ref{lemRv}, we relate first returns to $\g([v])$ in $H$ to first returns to $\g$ in $\cT_{\g}$, to obtain a value of $R_{\g([v])}$
given by
\beq{DRv}
R_\g=\ooi/(1-f),
 \eeq
where $f$ is a first return probability to $\g=\g([v])$ in the  arborescence  $\cT_\g$.
  Let $v$ be a nice vertex of $G$, i.e.
  $v$ is tree-like to   distance $\ell=\e \log_r n$. Thus any cycle containing $v$ has girth
  at least $2\ell$.
  Because the walk is non-backtracking,
  once it leaves $v$ it cannot begin to return to $v$,  until it has traveled
  far enough to change its direction, i.e. after at least $\ell$ steps. A direct return to $v$
  from a vertex $u$ at distance $\ell$, can be modeled as a biassed random walk, in which the walk succeeds only if it moves closer to $v$ at every step, with probability $1/(r-1)$. If this fails, the walk moves away from $v$ once more to distance $\ell$.
  Thus the probability of any return to $v$, and hence $\g([v])$ from distance $\ell$ during $T$ steps
  is given by $O(T/(r-1)^{\ell})=o(1)$.

In the case of $\g([v])$, $\g$  has no loops,  so the first return probability in  $\cT_{\g}$ is $f=0$.  This gives
  \[
  p_{\g}= \sooi (1-f) \wt \pi_{\g}=\sooi \frac 1n.
  \]
Applying Corollary \ref{geom} to $\g([v])$ in $H$ we have
\beq{wat}
 \Pr_{\cH}(\cA_{\g([v])}(t))=\ooi \exp (-\sooi t/n).
\eeq
To estimate the probability $\Pr_W(\cA_v(t))$, that $v$ is unvisited during $T,...,t$,
we use the equivalent walk $\cM$ in the digraph $M$, and contract $[v]$ to a vertex $\g=\g([v])$ to give a walk $\cH$ in $H$.
Using Lemma \ref{contract} with \eqref{wat}
 establishes the result that
\[
\Pr_W(\cA_v(t))=\Pr_{\cM}(\cA_{[v]}(t))=\ooi
 \Pr_{\cH}(\cA_{\g([v])}(t))=\ooi \exp (-\sooi t/n).
 \]
It follows that at step $t$ of $W$, the vacant set $\cR(t)$ is of expected size
\[
\E |\cR(t)| = \sum_{v \in V} \Pr_W(\cA_v(t))=
|\cN| e^{-\sooi t/n} + |\ol \cN|
\sim n e^{- \sooi t/n}.
\]
The concentration of $|\cR(t)| $ follows from
the methods of Lemma \ref{th4}. Theorem \ref{nobac}(i) for $|\cR(t)| $  follows from $\E |\cR(t)|$, the  concentration of $|\cR(t)| $ and the fact that
$o(n)$ vertices are not nice. Theorem \ref{nobac}(iii), for vertex cover time   follows from equating $\E |\cR(t)|=o(1)$ and applying the techniques used in \cite{CFReg} to obtain a lower bound.

{\bf Number of edges in the vacant set.}
The vertices $u,v$ are unvisited in $G$ if and only if the corresponding set of states $S=[u] \cup [v]$ is unvisited in $M$. Let $u,v \in \cR(t)$ and let $\{u,v\}$ be an edge of $G$ and hence of $\G(t)$. In this case, for nice $u,v$, the corresponding set of states $S$ of $M$ induces into two disjoint components given by
\begin{eqnarray*}
S_u&=& \{(u,v)\} \cup \{ (x,u), x \in N(u), x \ne v\}\\
S_v&=& \{(v,u) \} \cup \{ (x,v), x \in N(v), x \ne u\}.
\end{eqnarray*}
The total in-degree and out-degree of $S_u$ is $r(r-1)$. The details of the edges incident with e.g. $S_u$ are as follows.
The set $S_u$ induces $(r-1)$ internal edges in $M$ of the form
$((x,u),(u,v))$. For a state $e=(x,u) \in S_u$ there are $(r-1)$ states $f$ of $M$, $f=(a,x), x \ne u$ which point to
$e$, a total in-degree from $U \sm S$ to $S_u$ of $(r-1)^2$.
Similarly, $S_u$ points to $(r-1)+(r-1)(r-2)$ distinct states of $U$ not in $S$.
In total, the in-degree and out-degree of $\g(S)$ is $2r(r-1)$ of which $2(r-1)$ edges are loops at $\g(S)$. This means $2(r-1)^2$ states (other than $\g(S)$)
point to $\g(S)$.

We claim  $\wt \pi(\g(S))=2/n$, and
 that for $f \not \in S$, we have $\wt \pi(f)=1/rn= \pi(f)$.
We use $\wt \pi=\wt \pi \wt P$ to confirm this. For $\g(S)$ we have
\begin{flalign*}
\wt \pi(\g(S))&= \sum_{f \in N^-(\g(S))} \frac{\wt \pi(f)}{r-1}+
\wt \pi(\g(S)) \frac{2(r-1)}{2r(r-1)}\\
&= 2(r-1)^2 \frac{1}{rn} \frac{1}{r-1}+ \frac{2}{n} \frac{1}{r}=\frac{2}{n}.
\end{flalign*}
If $f \not =  \g(S)$, but $f \in N^+(\g(S))$ then
\begin{flalign*}
\wt \pi(f) &= \frac{\wt \pi(\g(S))}{2r(r-1)} + \sum_{e \in N^-(f) \atop
e \ne \g(S)}
\frac{\wt \pi(e)}{r-1}
\\
&=\frac{2}{n}\frac{1}{2r(r-1)}+\frac{1}{rn}\frac{r-2}{r-1}=\frac{1}{rn}.
\end{flalign*}
For any other state $f$, the relevant rows of $\wt P$ are identical with $P$
confirming $\wt \pi(f)=\pi(f)=1/rn$.
Hence for $f \ne \g(S)$, $\wt \pi(f)=\pi(f)$ so $\wt\pi_{\cH}$ is compatible with $\pi_{\cM}$ in Lemma \ref{contract}.

Consider next $R_{\g(S)}$. In the infinite arborescence $\cT_\g$ there are $(r-1)$ loops at $\g$ so $f=(r-1)/r(r-1)=1/r$. From \eqref{DRv} we obtain
\beq{p1}
p_{\g(S)}\sim \wt \pi(\g(S))(1-f)
= 2(r-1)/rn
.
\eeq
Using the
observation that at most $o(rn)$ edges of $\G(t)$ are incident with
vertices which are not nice ($v  \in \ol \cN$),
 the expected size of the edge set $E(\G(t))$ of the graph $\G(t)$ induced by the vacant set is
\[
\E(|E(\G(t))|) \sim \frac{rn}{2} e^{-\frac{2(r-1)t}{rn}\ooi}.
\]
This plus a concentration argument similar to Lemma \ref{th4},
 completes the proof of Theorem \ref{nobac}(i).

{\bf Number of paths length two in the vacant set.}
Let $u,v,w \in \cR(t)$ be such that $u, w \in N(v)$.
Thus $uvw$ is a path of length two in $G$ and hence $\G(t)$.
The assumption that $u,v,w$ are unvisited in $G$ is equivalent to
 $[u] \cup [v] \cup [w]$ unvisited in $M$. Let $S=[u] \cup [v] \cup [w]$.
The set $S$ can be written as
\begin{eqnarray*}
S&=& \{(u,v), (v,w), (w,v), (v,u)\} \cup \{ (x,u), x \in N(u), x \ne v\} \\
&\cup&
\{ (y,v), y \in N(v), y \ne w, u\}\cup\{(z,w), z \in N(w), z \ne v\}.
\end{eqnarray*}
Thus $S$ induces a single component in the underlying graph of $M$.
Counting the elements of the sets $S$ in the order above
we see that $S$ has size $4+(r-1)+(r-2)+(r-1)=3r$, and hence a total in-degree (resp. out-degree) of $3r(r-1)$. Of these edges, $2+(r-1)+2(r-2)+(r-1)=4(r-1)$ are internal.

We claim that $\wt \pi(\g(S))= 3/n$, and for $f \ne \g(S)$, $\wt \pi(f)=1/rn$.
We use  $\wt \pi = \wt \pi \wt P$ to confirm this.
For $\g(S)$,
\begin{flalign*}
\wt \pi(\g(S))&= \sum_{f \in N^-(\g(S))} \frac{\wt \pi(f)}{r-1}
+ \wt \pi(\g(S))\frac{4(r-1)}{3 r(r-1)}\\
&=\frac{1}{rn} \;\frac{3r(r-1)-4(r-1)}{r-1}+ \frac{3}{n}\;\frac{4(r-1)}{3r(r-1)}= \frac{3}{n}.
\end{flalign*}
For any state $f$ which is an out-neighbour of $\g(S)$,
there are $(r-1)$ parallel edges from $\g(S)$ to $f$. For example let $f=(w,x)$, $x \ne v$,
then states $ (z,w), z \ne x$ of $S$ point to $(x,w)$.
Thus
\[
\wt \pi(f)= \wt \pi (\g(S)) \frac{r-1}{3r(r-1)}=\frac{1}{rn}.
\]
We obtain that
$\wt \pi_{\cH}$ is compatible with $\pi_{\cM}$ in Lemma \ref{contract}.

To estimate $R_{\g(S)}$ consider $\cT_{\g(S)}$. The vertex $\g(S)$
has $4(r-1)$ loops and total out-degree $3r(r-1)$ giving a value for $f$
in \eqref{DRv} of
$f=4/3r$.  Thus
\beq{p2}
p_{\g(S)} \sim (3r-4)/rn.
\eeq

{\bf Threshold for the  vacant set.}
Theorem \ref{nobac}(iv) follows from using $Q=2M(2)-M(1)$ (see \eqref{better}), and equating $Q=0$ in Lemma \ref{mylem}
with the appropriate values of $\E M(1,t), \; \E M(2,t)$ as in \eqref{EMst}.
From \eqref{p1}, \eqref{p2} we have
$\a_1=2(r-1)/r,\; \a_2=(3r-4)/r$.
Equating $M(1,t)=2 M(2,t)$ and setting $t=u^* n$ gives
\[
u^*=\frac{r-2}{r} \log (r-1).
\]

\subsection{Properties of the vacant net}

{\bf Size of the vacant net.}
The calculations for the vacant net are much simpler than for the vacant set.
For the case  of an unvisited edge $\{u,v\}$ of $E(G)$, where $u,v$ are nice, the corresponding unvisited states of $U$ in $M$ are $S=\{(u,v),(v,u)\}$.
Contract $S$ to a vertex $\g(S)$.
The equations $\wt \pi= \wt \pi \wt P$ for the walk $\cH$ in $H$ are solved by $\wt \pi(\g(S))=2/rn$ for $\g(S)$,
and $\wt \pi(\s)=1/rn$ for any other state $\s$ of $H$. Thus $\wt \pi_{\cH}$  is compatible with $\pi_{\cM}$ in Lemma \ref{contract}.
No first return  to $\g(S)$ is possible in the arborescence $\cT_{\g(S)}$, and so $f=0$
in \eqref{DRv}. Thus
\beq{p3}
p_\g \sim 2/rn,
\eeq
 and the vacant net is of expected size
$\E |\cU(t)|\sim (rn/2) e^{-2t/rn}$.  The concentration of the random variable $|\cU(t)|$ follows from
the methods of Lemma \ref{th4}. Theorem \ref{nobac}(ii) follows from this.
The edge cover time in Theorem \ref{nobac}(iii) is obtained by equating  $\E |\cU(t)|=o(1) $ and applying the techniques used in \cite{CFReg} to obtain a lower bound.

{\bf The number of paths length two in the vacant net.}
Let  $\{v,u\},\{u,w\}$ be adjacent unvisited edges of $E(G)$. The corresponding states of $U$ in $M$ are $S=\{(u,v), (v,w), (w,v), (v,u)\}$ which contacts to
    a vertex $\g(S)$ with total in-degree and out-degree $4(r-1)$.
At $\g(S)$ there are two loops and $4r-6$ in-neighbours
    other than $\g(S)$. We obtain a stationary probability  $\wt \pi(\g(S))=4/rn$
and $\wt \pi(f)= \pi(f)$ for $f \not \in S$ which confirms $\wt \pi$ and $\pi$ are compatible.

 The total out-degree of $\g(S)$ is $4(r-1)$, but there are two loops at $\g(S)$ which can be chosen for a first return in $\cT_{\g(S)}$, with probability $2/4(r-1)$.  If the walk moves away from $\g(S)$, no first return is possible in $\cT_{\g(S)}$.
This gives $f=1/2(r-1)$  in \eqref{DRv}.
 Thus
 \beq{p4}
 p_\g\sim 2(2r-3)/r(r-1)n.
 \eeq

{\bf Threshold for the   vacant net.}
Theorem \ref{nobac}(v) follows from using $Q=2M(2)-M(1)$ (see \eqref{better}), and equating $Q=0$ in Lemma \ref{mylem}
with the appropriate values of $\E M(1,t), \; \E M(2,t)$ from \eqref{EMst}.
From \eqref{p3}, \eqref{p4} we have  $\a_1=2/r,\; \a_2=2(2r-3)/r(r-1)$.
Equating $M(1,t)=2 M(2,t)$ and setting $t=\th^* n$ gives
\[
\th^*= \frac{r(r-1)}{2(r-2)} \log (r-1).
\]

\section{Random walks which prefer unvisited edges} \label{edge-proc}

The unvisited edge process is a
 modified random  walk $X=(X(t),\; t \ge 0)$ on a graph $G=(V,E)$, which uses unvisited edges when available at the currently occupied vertex.
If there are {\em unvisited  edges} incident with the current
vertex, the walk  picks one u.a.r. and  makes a transition along this edge.
If there are no unvisited edges incident with the current vertex, the walk moves to a random neighbour.

\paragraph{Partitioning the edge-process into red and blue walks.}

At any step $t$ of the walk, we partition the edges of $G$ into
red (unvisited) edges and blue (visited) edges. Thus $t=t_R+t_B$ where $t_R$ is
the number of transitions along red  edges up to step $t$, hence recoloring those edges blue, and $t_B$ the number of transitions along blue edges.
Note that in \cite{BCF} the unvisited edges were designated blue and the visited edges red, the opposite of the terminology in this paper.

At each step $t$ the next transition is either along a red or blue edge. We speak
of the sequence of these edge transitions as the red (sub)-walk and the blue
(sub)-walk. The
walk thus consists of red and blue phases which are maximal sequences of edge transitions of the given edge type (unvisited or visited).
For any vertex $v$, and step $t$, let $d_B(v,t)$ the blue  degree of
$v$, be the number of blue edges incident with $v$ at the start of  step $t$.
Similarly define $d_R(v,t)$.

For  graphs of even degree,
each red phase starts at some step $s$ at a vertex $u$
of positive even red degree $d_R(u,s) \ge 2$, and ends at some step $t$ when the
walk returns to $u$ along the last red edge incident with $u$. Thus $d_R(u,t)=0$
and a blue phase begins at step $t+1$.
Thus for $r$-regular graphs  $r=2d$, if we ignore
the red phases of the edge-process $X$, then the resulting blue phases describe a
 simple random walk $W$ on the graph $G$.
To illustrate this, suppose the edge-process $X$ starts at $X(0)=u$, then $W$  also
starts at vertex $u$ after the completion of the first red phase at $t_R$.
After some number of steps $t_B$, the blue walk $W$ arrives at a vertex $u'$ with unvisited edges, and a red phase starts from  $u'$, at step $t_R+1$, as counted in the red walk. This is followed by a  blue phase starting from $u'$
at step $t_B+1$ of the blue walk.
Thus the walks interlace seamlessly, and at step $t$ of the edge-process, we have $t=t_R+t_B$, where $t_R, \; t_B$ are the number of red and blue edge transitions.

In summary the red walk is a walk with jumps which consists of a sequence of
closed tours each with a distinct start vertex. The blue walk is a simple random walk. Given a step $s=t_R+t_B$ of the edge-process, we extend the notation $d_R(u,s)$ for the red degree of vertex $u$ at step $s$ of the edge-process to $d_R(u,t_R)$ the red degree of vertex $u$ at step $t_R$ of the red walk.

\subsection{Thresholds measured in the red walk}\label{redstep}

To make our analysis, we first consider only the red walk steps $t=t_R$. Let $r=2d$ and let $R_j(t)$
be the number of vertices of red degree $j$ for $j=0,1,...,2d$
at step $t$ of the red walk. Unless the walk is at
the  vertex $u$ which starts the red phase, (in which case all vertices have even red degree), then with the exception of $u$ and the current position $v$ of the walk, all other vertices have even red degree at any step of a red phase.

We generate the red walk in the configuration model, and
derive its approximate  degree sequence. The intuition is as follows.
Suppose the red walk arrives at vertex $v$ at the end of step $t$, and leaves $v$ at the start of step $t+1$. To simplify things we could agree to say the degree
of $v$ changes by 2 at the start of step $t+1$.
Thus we consider the following process which samples u.a.r. without replacement from the sets $S_v, v \in V$ of configuration points of a graph with $m=dn$ edges.

\begin{quote}
{\bf Pairs-process.}\\
At each step $t=1,...,m$:\\
 Pick an unused configuration point $\a$ u.a.r., remove $\a$ from the set of available points.
 Pick  another unused point u.a.r. $\b$ from the same vertex as $\a$,  remove $\b$ from the set of available points.\\
Add
$Y_t=(\a,\b)$ to the ordered list of samples $Y_1,...,Y_{t-1}$.
\end{quote}

Let the random variables $\ul N_k(t), \; k=0,1,...,d$, be the number of vertices of degree $2k$ generated by the Pairs-process, and let $N_k(t)= \E \ul N_k(t)$. Here the degree of a vertex is the number of {\em unpaired} points associated with that vertex.

We condition on the pairings in our process and the ordering $\a,\b$ within pairs. After this, we have a permutation of $dn$ objects, where each object is a pair.
  The probability $p(k)$ that a vertex contributes to $\ul N_k(t)$ is the probability that exactly $d-k$ out of a fixed set of $d$ objects appear before the $t$-th element in our permutation. Thus $p(k)$ has a hypergeometric distribution, and
\[
N_k(t)= np(k)= n \frac{{t \choose d-k}{dn-t \choose k}}{{dn \choose d}}.
\]
Thus $N_k(t)$ can be written as,
\begin{equation}\label{answer}
N_{k}(t) = \brac{1+O\brac{{\frac{1}{t}+\frac{1}{dn-t}}}} \; n \; {d \choose k} \bfrac{dn-t}{dn}^{k} \bfrac{t}{dn}^{d-k},
\quad i=0,1,...,d.
\end{equation}

\bignote{(AF writes: Here is a simpler and complete justification for \eqref{answer}.\\
Nice. Thanks. Makes life much easier.\\
When I looked at what you wrote, I realized it was Hypergeometric, which makes things nice and natural and easy
}

By a martingale argument on the configuration sequence of pairs of points of length $dn$, the random variables $\ul N_k(t)$ are concentrated within
$O(\sqrt{dn \log n})$ of $N_k(t)$ for any $0 \le t \le dn$.
Suppose the first difference between a pair of sequences $Y,Y'$ occurs at
vertices $v, v'$ with $Y_i=(\a_i,\b_i)$ and $Y'_i=(\a'_i,\b'_i)$.
Let $Y_j$ be the
first occurrence of $v'$ after step $i$ in $Y$. Map this to the first occurrence $Y'_j$ of $v$ in $Y'$. For $u \ne v,v'$ let all other entries of the sequence be the same. Map subsequent pairings $Y_\ell, Y'_\ell$  between (possibly) different configuration points of $v$  as appropriate; and similarly for $v'$.
The maximum difference between $\ul N_k(t)$ in the mapped sequences is 2.
Thus
\beq{martN}
\Pr(|\ul N_k(t)-N_k(t)| \ge \sqrt{A n \log n}) =O(n^{-A/8}).
\eeq
We next explain how  $N_k(t)$ can be used to approximate  $R_{2k}(t)$, the  number of vertices of red degree $j=2k$.
Let $Y$ be a Pairs-process and $W$ a red walk  generated in the configuration model. Let the vertices  which start the red phases of $W$ be $\ul u=(u_1,u_2,...,u_J)$.  There is an isomorphism between $(W, \ul u)$ and $(Y, \ul u)$.  Let $Y_i=(\a_i,\b_i)$ be the pair generated at step $i$ of Pairs. Let $Y_\ell$ be the last occurrence in $Y$, of configuration points from vertex $u=u_1$ (i.e.from $S_u$). The subsequence $P=(Y_1,\cdots,Y_\ell)$ of Pairs is isomorphic to the first red phase $Q$ of $W$  by the following mapping which moves $\b_\ell$ to the front of $P$ to form $Q$.
\begin{flalign*}
P = &\;\; (\a_1,\b_1), (\a_2,\b_2), \cdots, (\a_{i-1},\b_{i-1}),(\a_i, \b_i), \cdots, (\a_\ell,\b_\ell),\\
 Q= & \;\;(\b_\ell, \a_1,(\b_1, \a_2), \cdots, (\b_{i-1},\a_i),(\b_i, \a_{i+1}), \cdots, (\b_{\ell-1},\a_\ell).
\end{flalign*}
Given $W$ and sequence $\ul u=(u_1,...,u_J)$, the last occurrence of $u_i$ is before the last occurrence of $u_{i+1}$. Thus there is always a unique $Y$ to match this $W$.

We next relate the probability of  a given $(W, \ul u)$ to that of  the corresponding $(Y, \ul u)$.
Let $v$ be the vertex chosen to pair at step $i$ of $Y$. In the Pairs-process, let $d(v,i)$ be the
number of remaining unmatched configuration points of $v$ at the start of step $i$. The total degree of the underlying graph is $2m$.
Thus
\begin{align*}
\Pr(Y_1)&= \frac{1}{2m}\; \frac{1}{d(v,1)-1}\\
\Pr(Y_i \mid Y_1,\cdots, Y_{i-1}) &=\frac{1}{2m-(2i-2)} \;\frac{1}{d(v,i)-1}
.
\end{align*}
When $v \ne u$, and the transition is $W_{i+1}=(\b_i,\a_{i+1})$, the vertex $v$
which corresponds to $\a_i$ in $Y_i=(\a_i, \b_i)$ had degree $d(v,i)$ when  $\a_i$ was chosen u.a.r.,  and so $\b_i$ was chosen from a set of size $d(v,i)-1$ and so
\[
\Pr(W_{i+1} \mid W_1,\cdots, W_i)= \frac{1}{d(v,i)-1} \;\frac{1}{2m-(2i+1)}.
\]
However if $v=u$,
because we moved $\b_\ell$
to the front of $W$ the red degree of $u$ at step $i$ is less by one than it was in Pairs.
Thus
\begin{align*}
\Pr(W_1)&= \frac{1}{d(u)}\; \frac{1}{2m-1}\\
\Pr(W_{i+1} \mid W_1,\cdots,W_i)&= \frac{1}{d(u,i)-2}\;\frac{1}{2m-(2i+1)}.
\end{align*}
This means that, at step $\ell$ when the red degree of $u=u_1$ becomes zero,
\[
\Pr_W(Q) = \Pr_Y(P) \frac{(d(u)-1)(d(u)-3) \cdots 1}{d(u)(d(u)-2)\cdots 2}
\prod_{i=0}^{\ell-1}\frac{2m-2i}{2m-2i-1}.
\]
We  repeat this analysis starting with $u=u_2$ etc. Thus
with
$\Pr(Z)$ being the probability of process $Z$,
\[
\Pr(W)
 \le \Pr(Y) \prod_{i=0}^{m-1}\brac{\frac{2m-2i}{2m-2i-1}}
 =\Pr(Y) \frac{(2^mm!)^2}{(2m)!}
= O(\sqrt{m}) \;\Pr(Y) .
\]
Recall that $R_{2k}(t)$ is the number of vertices at step $t$ of the red walk.
Suppose we generate a red walk starting from $u$ in the configuration model, stopping  at step $j$
 to give
 $Q=(a_1,b_1,a_2,b_2, \cdots, a_j,b_j)$.
Then $P=(b_1,a_2),\cdots, (b_{j-1},a_{j})$ is a Pairs sequence, and for any $j$
\[
R_{2k}(j)=\ul N_k(j-1)+C, \qquad |C| \le 2.
\]
Using \eqref{martN}  with $A=24$, we have,
\beq{bigtime}
\Pr(\exists t, \; R_{2k}(t) \ge |N_k(t) + O(\sqrt{ n \log n})|) \le  O(n\sqrt{dn} \; n^{-A/8})=O(n^{-1}).
\eeq
Using $t_R$ to denote red steps we can obtain the size of the vacant set $R_d(t_R)$. We do this next. Theorem \ref{newedge} is expressed in terms of
step $t$ of the
 Edge-process, where $t=t_R+t_B$ and $t_B$ are blue steps.
Thus, to prove Theorem \ref{newedge} we need to add back the number of blue steps $t_B$.
We do this in Section \ref{realproof}.

\paragraph{Vacant set properties at any step of the red walk.}
Let $t=dn(1-\d)$ where $\d>0$ constant. Then $N_d(t)= \Th(n)$, and  w.h.p.
the size of the vacant set at red step $t$ is
\beq{vacs}
|\cR(t)|= R_{2d}(t)= \ooi N_d(t)=\ooi n \bfrac{dn-t}{dn}^{d}.
\eeq
Let $M(1,t)$ denote  the  expected number of edges (resp. $M(2,t)$ denote twice the expected number of pairs of edges)
induced by the vacant set at each vertex of the vacant set.
Working in the configuration model, with $r=2d$,
\begin{eqnarray}
M(1,t) &\sim& N_d(t) r \frac{2d}{2dn-2t} N_d(t) \sim rn
\bfrac{dn-t}{dn}^{2d-1}\label{M1set}\\
M(2,t) & \sim & N_d(t) {r \choose 2}\brac{\frac{2d}{2dn-2t} N_d(t)}^2 \sim
{r \choose 2}n \bfrac{dn-t}{dn}^{3d-2}.\label{M2set}
\end{eqnarray}
The expected number of edges induced by  the vacant set is
\beq{nacs}
\E|E(\G(t))|\sim M(1,t)/2 \sim dn \bfrac{dn-t}{dn}^{2d-1}.
\eeq
The concentration of $M(1,t),\;M(2,t)$
follow from the  methods of Lemma \ref{th4} (the Chebychev inequality and the interpolation).

The threshold $t=t^*$ for the subcritical phase
comes from applying the Molloy{\red -}Reed  condition
given by
$L(\bd)=0$.
In \eqref{newLd}  we examine
$2M(2,t)-M(1,t)-O(|\ol \cN|)$, where $M(1,t), M(2,t)$ are given by
\eqref{M1set}-\eqref{M2set}, and $|\ol \cN|=O(n^{\e})$ is the number of non-nice vertices (see \eqref{nice}). Let $t^*= u^* n$,  where
\beq{rs}
u^*=  d \brac{1- \bfrac{1}{2d-1}^{\frac{1}{d-1}}}.
\eeq
 Note that $M(1,t^*)= \Th (n) , M(2,t^*) = \Th (n)$. At $t^*$,  $2M(2,t*) \sim M(1,t*)$.
If we choose $t=u^*n(1+\ve)$, where  $|\ve|> 0$  constant,
then using \eqref{M1set}-\eqref{M2set}
 gives
\beq{emm2}
2M(2,t)-M(1,t)= |\Th(n)| \brac{(1-\ve((r-1)^{1/(d-1)}-1)^{d-1}-1}.
\eeq
Thus for $t = t^*(1-\ve)$ this difference is positive. As $|\cR(t)|=\Th(n)$, w.h.p. this
confirms the w.h.p. existence of a giant component linear in the graph size.
At $t=t^*(1+\ve)$ the difference in \eqref{emm2} is negative, and the maximum component size is $O(\log n)$. It remains to find out how many blue steps have elapsed by red step  $t^*(1+\ve)$.
We defer this until Section \ref{vacnet}.

\paragraph{Vacant net properties at any step of the red walk.}

The vacant net has exactly $U(t)=dn-t$ edges. Thus, similarly to the vacant set,
\begin{eqnarray}
M(1,t)&\sim & 2dn-2t \label{M1}\\
M(2,t) & \sim& \sum_{k=1}^d {2k \choose 2} N_k(t)\sim \frac{dn-t}{dn}(dn+2(d-1)(dn-t)).\label{M2}
\end{eqnarray}
In \eqref{newLd}  for the Molloy-Reed condition we require
$2M(2,t)-M(1,t)-O(|\ol \cN|)>0$, where $M(1,t), M(2,t)$ are given by
\eqref{M1}-\eqref{M2}, and $|\ol \cN|=O(n^{\e})$ is the number of non-nice vertices (see \eqref{nice}).

The solution to $M(1,t)= 2M(2,t)$  obtained by using the right hand side
values of \eqref{M1}-\eqref{M2} is at the end of the red walk, i.e. red step $\wh t=\th^*n$ where
\[
\th^*=d.
\]
For any red step $t(\d) =(1-\d)dn$ where $\d >  0$,
$M(1,t)=\Th(n), M(2,t)=\Th(n)$, and
\[
2M(2,t)-M(1,t)\sim 4 (d-1) \frac{(dn-t)^2}{dn}.
\]
Thus at red step $t(\d)=(1-\d) \wh n$,  for any $\d>0$, w.h.p. the vacant net has a giant component
linear in the graph size.
It remains to find out how many blue steps $t_B$ have elapsed before this value of $t=t_R$, and also to analyse  the sub-critical case $dn-t=o(n)$. We defer this until Section \ref{vacnet}.

\subsection{Number of blue steps before a given red step}\label{numblue}
Suppose a red phase starts at red step $s$ from vertex $v$ of red degree $2k$.
At step $s$ the walk leaves $v$ along a red edge, and returns to $v$ at some step $t' \ge s$.
We have  $d_R(v,\t)=2k-1$
for $s \le \t <t'$ and $d_R(v,t')=2k-2$. Thus a red phase at $v$ consists of $k$
{\em excursion rounds} with starts $s_1,...,s_k$ and ends $t_1,...,t_k$, where $s_1=s,t_k=t$ and $s_i=t_{i-1}+1$. At the final return, $d_R(v,t)=0$ and a blue phase begins.

\begin{lemma}\label{Lsk}
Let $L(s,k)$ be the finish time of a red phase starting at red step $s$
from a vertex $v$ of red degree $2k$.
Let $r=2d$. Then for $t\le dn(1-\d)$, and  $\d\ge\om/\sqrt{n}$
\beq{Lskt}
\Pr(L(s,k)=t) \le (1+O(k/\om)) \frac{k}{2^{2k}}{2k \choose k} \frac{(t-s)^{k-1}}{
(dn-s)^{k-1/2}(dn-t)^{1/2}}.
\eeq
\end{lemma}
\begin{proof}
Let $\r=2k-1$. For a walk starting from $v$ at $s$, let $T^+_v$ be the first return time to $v$. Then working in the configuration model,
\begin{flalign*}
\Pr(T^+_v=t \mid s, 2k)&=\prod_{\s=s}^{t-1} \brac{1-\frac{\r}{2dn-2\s-1}} \frac{\r}{
2dn-2t-1}\\
&= \frac{\r}{2dn-2t-1}
\exp\brac{-\frac{\r}{2}\brac{
\frac{1}{dn-s}+\cdots+\frac{1}{dn-t}}+
O\bfrac{t-s}{(dn-t)^2}}\\
&= \frac{\r}{2dn-2t-1}
\exp\brac{\frac{\r}{2}\brac{\log\frac{dn-t}{dn-s}+O\bfrac{1}{n\d^2}}}\\
&= (1+O(1/\om)) \frac{\r}{2(dn-t)} \bfrac{dn-t}{dn-s}^{\r/2}.
\end{flalign*}
Thus
\begin{flalign*}
\Pr(L(s,k)=t)&= \sum_{s<t_1<\cdots<t_{k-1}<t}\prod_{i=1}^k \Pr(T_v^+=t_i \mid s_i,\r_i=2k-2(i-1)-1)\\
&= (1+O(1/\om))^k ((2k-1)(2k-3)\cdots 1 )\frac{1}{2^k} \sum_{s<t_1<\cdots<t_{k-1}<t}
\prod_{i=1}^k \frac{1}{dn-t_i}\bfrac{dn-t_i}{dn-s_i}^{\r_i/2}\\
&=(1+O(k/\om))
\frac{(2k)!}{k! 2^{2k}} \frac{1}{(dn-s)^{k-1/2}}\frac{1}{(dn-t)^{1/2}}
\sum_{s<t_1<\cdots<t_{k-1}<t} 1\\
&= (1+O(k/\om)) {t-s \choose k-1} \frac{(2k)!}{k! 2^{2k}} \frac{1}{(dn-s)^{k-1/2}} \frac{1}{(dn-t)^{1/2}}
.
\end{flalign*}
\end{proof}
We  use Lemma \ref{Lsk} to upper bound the number of red phases before a given red step $t$.
\begin{lemma}\label{redlem}
Let $J(t)$ be the number of red phases completed at or before step $t=t_R$ of the red walk.
For $\d >0$ and any $t \le dn(1-\d)$, and $\d\ge \om/\sqrt{n}$
\[
\Pr\brac{J(t) \ge \frac{de^{3}}{\d}}= O\!\bfrac1n.
\]
\end{lemma}
\begin{proof}
The $J$ red phases start at $s_1,\cdots,s_J$ and end at $t_1,\cdots,t_J$, where $s_1=0$,
$t_J=t$, and $s_i=t_{i-1}+1$. The total number of excursion rounds is $K=k_1+\cdots+k_J$, where $J <K\le dJ$. Let $\ul \t=(t_1,...,t_J)$ and $\ul \k
=(k_1,...,k_J)$. Let $\cE(\ul \t, \ul \k)$ be the event that
$L(s_i,k_i)=t_i, \; i=1,...,J$. Then
\[
\Pr(\cE(\ul \t, \ul \k))= \prod_{i=1}^J \Pr(L(s_i,k_i)=t_i).
\]
To simplify \eqref{Lskt}, note that, as ${2k \choose k} \le 2^{2k}$ and $k \le d$,
\[
\frac{k}{2^{2k}}{2k \choose k} \le k \le d.
\]
Let $s=adn,t=bdn$ where $0 \le a \le b\le (1-\d)$. Then, as $(t-s)\le (dn-s)$,
\[
\frac{(t-s)^{k-1}}{
(dn-s)^{k-1/2}(dn-t)^{1/2}} =\bfrac{t-s}{dn-s}^{k-1}\!\! \frac{1}{(dn-s)^{1/2}
(dn-t)^{1/2}}\le \frac{1}{dn \d}.
\]
Thus
\beq{prekly}
\Pr(\cE(\ul \t, \ul \k))\le   (1+O(K/\om))\frac{d^J}{(dn)^J \d^J}.
\eeq
For a given realization of the edge-process, the sequence $\ul \k$ is a fixed input determined by the blue walk, and the right hand side of \eqref{prekly} is independent of the value of this input.
For any given red step $t_J=t$, let $J(t)=J(t,\ul \k)$ be the number of red phases 
completed at or before step $t$. Then,
\begin{flalign*}
\Pr(J(t), \ul \k) &= \sum_{\ul \t}\Pr(\cE(\ul \t, \ul \k))
\le (1+O(K/\om)){t \choose J-1} \frac{1}{n^J \d^J}\\
&=O(1+O(K/\om)) \frac{t^{J-1}e^J}{n^J \d^J J^J}= O\!\bfrac{1}{n} \bfrac{de^{1+O(d/\om)}}{J\d}^J.
\end{flalign*}
The last line follows from $K \le dJ$ and $t \le dn$. Finally, we put $J=de^{2}/\d$.
\end{proof}
\bignote{$\Pr(J(t), \ul \k)$ needs to changed to something meaningful.\\
Ok have done this}
\subsection{Proof of Theorem \ref{newedge}}\label{realproof}\label{vacnet}

Let $T_G$ be a mixing time for a random walk on $G$,
 such that, for $t\geq T_G$,
\begin{equation}\label{44}
\max_{u,x\in V}|P_{u}^{(t)}(x)-\pi_x| \le \frac{1}{n^3}.
\end{equation}

We use the following results from \cite{BCF}.
\begin{lemma}\label{crude}
Let $T_G$ be a mixing time of a random walk $W_u$ on $G$ satisfying
\eqref{44}. For any start vertex $u$ let $\ul A_{t,u}(v)$ be the event that
$W_u$ has not visited vertex $v$ at or before step $t$. Then
\[
\Pr(\ul A_{t,u}(v)) \le  e^{-\rdown{{t}/{(T_G+3\E_{\pi}H_v)}}},
\]
where $\E_{\pi}H_v$ is the hitting time of $v$ starting from stationarity.
\end{lemma}
It follows from \eqref{Tlogn} that we can take $T_G=120\log n$.
\begin{lemma}\label{HGam}
Let $G=(V,E)$, let $|E|=m$. Let $S \seq V$, and let $d(S)$ be the degree of $S$. Then $\E_{\pi} H_{S}$,
the expected hitting time of $S$ from stationarity satisfies
\[
\E_{\pi} H_{S} \le \frac{2m}{ d(S)(1-\l(G))}.
\]
\end{lemma}
In \eqref{nice2} we used the crude bound $\l_2 \le 29/30$.
Using this in Lemma \ref{HGam} gives an upper bound $\E_{\pi} H_{S} \le 30 n/|S|$. Lemma \ref{crude} implies that
\beq{2lemmas}
\Pr\ul A_{t,u}(v)) \le  \exp\brac{-\frac{t}{120(n/|S|+\log n)}}.
\eeq

\paragraph{Various useful properties.}
\begin{lemma}\label{useful}
\begin{enumerate}[(i)]
\item \label{i} Let $S(t)$ be the vertex set of the vacant net at red step $t$. Then for any
$t=dn(1-\d)$, $|S(t)| \ge \d n/2$.
\item \label{ii} Let $t_R=dn(1-\d)$ where  $\sqrt{(\om \log n)/n}\le \d=o(1)$.Then
w.h.p. by red step $t_R$ at most $t(\d)=dn(1-\d +O(1/\om))$ steps of the edge-process have elapsed.
\item \label{iii} There exists red step $t_1=dn(1-\d)$ with
 $\d=\Om(\sqrt{(\om \log n)/n})$,
 such that w.h.p. by red step $t_1$,
for $k \ge 3$, $R_{2k}(t_1)=0$, and $R_4(t_1)=O(\om \log n)$. The corresponding step of the edge-process is $t=dn(1+O(1/\om))$.
\end{enumerate}
\end{lemma}
\begin{proof}
{\em Part (\ref{i}).}
At  red step $t=dn(1-\d)$ there are, by definition, $\d dn$ red edges.
Let $S(t)=\{v \in V: d_R(v,t)>0\}$ denote the vertex set of the vacant net at red step $t$. Then deterministically
\[
|S(t)| \ge \d dn/r= \d n/2.
\]
{\em Part (\ref{ii}).}
At $t=dn(1-\d)$ the vertex set of the vacant net is of size $|S(t)|\ge \d n/2$. Let $S=S(t)$ in Lemma \ref{HGam}. Contract $S(t)$ to a vertex $v(S)$. Let $\t$
denote a step of the blue walk, and apply Lemma \ref{crude} at  $\t= A\log n/\d$ for some large $A$.
Using \eqref{2lemmas}, and choosing $A=400$, the probability $P_B(\t)$ that a blue phase lasts more than $T_G+\t\leq 2\t$ steps is upper bounded by
\[
P_B(\t) \leq\exp\set{-\frac{2A\log n/\d}{120(\log n+2/\d)}}=O(1/n^2).
\]
Let $t_R$ be the end of the first red phase at which $|S(t)| \le n \d$.
Let $t_B$ be the number of blue steps before $t_R$, then
\[
t_B=\sum_{i \le J} t_{B,i}\leq  2J(t_R)\t= O \bfrac{\log n}{\d^2}.
\]
Thus provided $\d \ge \sqrt{\om\log n/n}$, $t_B=O(n /\om)$ and
\[
t(\d) =t_R+t_B= dn(1-\d)+ O(n/\om)=dn (1-\d +O(1/\om)).
\]
{\em Part (\ref{iii}).}
Let $t=dn(1-\d_0(k))$ where $\d_0(k)= (\log n/n)^{1/2k}$. At $t_R=t$,  \eqref{answer} and \eqref{martN} imply that  w.h.p.
\[
R_{2k}(t)= O(n \d_0^k) +O(\sqrt{n \log n})=O(\sqrt{n \log n}).
\]
Next choose $\d_1=\sqrt{\om \log n/n}$. Let $B_{2k}(t)$ be the set of vertices of red degree $2k$ at red step $t$.
Let $t'=t+dn\d_1$ and let $P_B(v)$ be the probability the red walk did not visit $v$ during $ dn\d_1$ steps. Thus
\begin{align*}
P_B(v) & = \prod_{s=t}^{t'} \brac{1-\frac{k}{dn-s}}
= O(1) \exp \brac{-k \log \frac{dn-t}{dn-t'}}\\
&=O(1) \exp\brac{-k \log \d_0/\d_1}=O\brac{\om^{1/2} (\log n/ n)^{(k-1)/2}}.
\end{align*}
Thus for $k=d,d-1,...,3$,
\[
\Pr(R_{2k}(t') \ne 0) = O\brac{\sqrt{\om n \log n} \;(\log n/n)^{(k-1)/2}}=o(1).
\]
For $k=2$
\[
\E R_4(t') = O(\sqrt{\om}\log n),
\]
and thus
\[
\Pr(R_4(t') \ge \om \log n)=O(1/\sqrt{\om}).
\]
By the previous part of this lemma, the number of blue steps  $t_B$ elapsed at $t'$ is $O(n/\om)$. This corresponds to a step $t$ of the edge-process where
\[
t=t'+t_B \le dn(1+O(1/\om)).
\]
\end{proof}

\paragraph{Vacant set size and  threshold.}
We recall the discussion in Section \ref{redstep} where the size, and number of edges of the vacant set at any red step $t_R$ are given by \eqref{vacs} and \eqref{nacs} respectively. Theorem
\ref{newedge} (i) then follows from Lemma \ref{useful}\eqref{ii}.

Considering the threshold, let $t^*=u^*n$ be the red step given by $u^*$ in  \eqref{rs}.
We prove that at steps $t^*(1-\e)$ and $t^*(1+\e)$ respectively of the edge-process, the vacant set is super-critical and sub-critical respectively.
At red step $t^*$,
\[
|\cR(t^*)|=R_{2d}(t^*)= \ooi n \bfrac{1}{2d-1}^{\frac{d}{d-1}}= \Theta(n),
\]
and $|\cR(t)|$ is concentrated.
In Section \ref{redstep},   using the Molloy-Reed condition and \eqref{emm2}, we proved that at  red step $t\le t^*(1-\e)$ the giant component $C_1(t)=\Th(|\cR(t)|)=\Th(n)$ w.h.p.
Similarly at red step $t^*(1+\e)$ for some small $\e>0$, the maximum component size of the vacant set at $t^*(1+\e)$ is $O(\log n)$ w.h.p.
Let $d(1-\d)=u^*(1+\e)$, then the corresponding $\d$ is constant. By Lemma \ref{useful}(\ref{ii}), red step $t_R=t^*(1+\e)$
corresponds to step $t=t^*(1+\e+O(1/\om))$ of the edge-process.
Thus at step $t=t^*(1+O(\e))$
the maximum component size is $O(\log n)$ w.h.p. and the
graph of the vacant set is subcritical.
This completes the proof of Theorem
\ref{newedge}  (iii).

\paragraph{Vertex cover time.}
For the proof of Theorem
\ref{newedge}  (iii), that w.h.p. $T_{cov}^V(G) \sim dn$, we consider the cases $r=4$ and $r \ge 6$ separately.

{\em Case $r=2d,\; d \ge 3$.}
At red step $t_R=dn(1-\d)$ where
$\d=1/ n^{1/2d}$, then $R_{2d}(t_R)=\Om(n^{1-1/2d})$.
By Lemma \ref{useful}(\ref{ii}) the corresponding step of the edge-process is
$t'=dn(1+O(1/ \om))$. However by Lemma \ref{useful}(\ref{iii}), at step $t_1=dn(1+O(1/\om))$ of the edge-process $R_{2d}(t_1)=0$ and the vacant set is empty. Thus the vertex cover time $T_{cov}^V(G) \sim dn$.

{\em Case $r=4$.} The cover time can be deduced from the proof of Lemma
\ref{24T} (see below) that $\wh t \sim dn$ is the threshold for the vacant net. The relevant facts from Lemma \ref{24T} are the following. At $t=\wh t (1-o(1))$
  there are vertices of red degree 4 w.h.p. For some $t\le \wh t \ooi$ the
last vertex of red degree 4 disappears. Thus, for $r=4$ the vacant set becomes empty at some $t \sim \wh t {\red \sim} dn$. We remark that there could still be  some isolated red cycles, in which case the vacant net is nonempty.

\bignote{Could there be some red cycles hanging on?\\
Indeed, and it is hard to  (provably) get rid of them, but they dont affect the vertex cover time}

\paragraph{Vacant net. Supercritical regime.}

From Section \ref{redstep} the threshold for the vacant net is at
 red step $\wh t \sim dn$.
Choose a red step $t_R$, where
$t_R=dn(1-\d)$, $\d \ge 0$ constant.
By Lemma \ref{useful}\eqref{i},  the vertex set $S(t_R)$ of the vacant net is of size  $|S(t_R)| \ge \d n/2$.
By Lemma \ref{useful}\eqref{ii},
the corresponding step $t=t_R+t_B$ of the edge-process is $t_R\ooi$ w.h.p.

\paragraph{Vacant net. Subcritical regime.}

Because the vacant net becomes sub-linear in size near $dn \sim \wh t$,
the time taken by the blue walk to reach unvisited edges increases rapidly.
Thus more work is needed to prove the vacant net
has maximum component size $O(\log n)$ at some step $t=dn\ooi$
of the edge-process.

\begin{lemma}\label{24T}
There is a step $t$ of the edge-process,
where $t=dn(1+o(1))$ such that w.h.p. at step $t$
all components of the vacant net have size $O(\log n)$.
\end{lemma}
\begin{proof}
The proof is in three parts. In the first part we count up the number of blue steps occurring before  red time $t_1=dn(1-\d_1)$ where $\d_1=\sqrt{\om \log n/n}$. At $t_1$ the vacant net  consists mainly of vertices of red degree 2, with a few vertices of red degree 4. In the second part, we prove that after a further $t_B=o(n)$ steps of the blue walk we have  removed all vertices of red degree 4, thus destroying any complex components of the vacant net.
The vacant net now  consists entirely of red cycles.
In the third part we use a further $t_B=o(n)$ steps of the blue walk to remove
any red cycles of length at least $\log n$.

\paragraph{Part 1.}

Let $t_1$ be red step $dn(1-\d_1)$ where $\d_1=\sqrt{\om \log n/n}$. By Lemma \ref{useful}(\ref{ii}) the corresponding step of the edge-process is
$t=d n (1+O(1/\om))$.
At any red step $t_R=dn(1-\d)$, the maximum component size is at most the number of red edges $dn\d$.
Thus at step $t$ of the edge-process corresponding to $t_1$ the giant component is of size
\[
C_1(t)=O( n\d_1)=O(\sqrt{\om n \log n}).
\]
By Lemma \ref{useful}(\ref{iii})
the vacant net $\wh\G(t_1)$
consists of $R_{2i}(t_1)=n_{2i}$ vertices of red degree $2i$.
For some $c_2$ constant, w.h.p.
\beq{nii}
n_2= c_2  \sqrt{n \om \log n},\qquad n_4 \le \om \log n,
\qquad n_{2i}=0, \;\; i\ge 3.
\eeq

\paragraph{Part 2.}
It follows from \eqref{nii} that at red step $t_1$
the vacant net consists of 2-cycles (cycles with vertices of red degree 2) and complex components  with vertices of degree 2 and 4.
Such components are Eulerian, and can be decomposed (non-uniquely) into  $(2,4)$-cycles (cycles where all vertices have red degree 2 or 4 in the vacant net). We  prove that after $t_B=o(n)$ further blue steps,
the blue walk has visited every $(2,4)$-cycle in the vacant net $\wh \G(t_1)$. If so,
the vacant net is either empty or consists entirely of red 2-cycles. To assume otherwise leads to a contradiction.

We count $(2,4)$-cycles in the configuration model. Let $\F(m)=(2m)!/m!2^m$.
Using ${2k \choose k}/2^{2k}=\Theta(1/(1+\sqrt{k}))$, it follows that
\beq{good1}
\frac{m!}{(m-s)!}2^s \frac{\F(m-s)}{\F(m)}=\Th\brac{\sqrt{\frac{m}{m-s+1}}}.
\eeq

Let $C(i,a,b)$ be the number of $(2,4)$-cycles of length $i=a+b$ and containing
$a$ vertices of red degree 2, and $b$ vertices of red degree 4. 
Thus
\[
\E C(i,a,b) = {n_2 \choose a}{n_4 \choose b}\frac{(i-1)!}{2}
{4 \choose 2}^b 2^i \frac{\F(n_2+2n_4-i)}{\F(n_2+2n_4)}.
\]
Let $m=n_2+2n_4$, $s=i$ in \eqref{good1}. Then, 
\begin{flalign*}
\E C(i,a,b) &= \Th\brac{\sqrt{\frac{n_2+2n_4}{n_2+2n_4-i}}}
\frac{1}{i} {i \choose b} 6^b
n_4^b
\frac{(n_2)_a}{(n_2+2n_4)_i}\\
&= 
\Th\brac{\sqrt{\frac{n_2+2n_4}{n_2+2n_4-i}}}
\frac{1}{i} {i \choose b} 6^b
\bfrac{n_4}{n_2+2n_4}^b e^{O\bfrac{(a+b)b}{n_2+2n_4}} \\
&= \Th\brac{\sqrt{\frac{n_2+2n_4}{n_2+2n_4-i}}}\frac{O(1)}{i}
\brac{\frac{ie^{1+O(a/n_2)}}{b}\frac{6n_4}{n_2}}^b.
\end{flalign*}

\bignote{as it said before,' after some work', which  was a bit more than what was given, so I put the details in. \\It needs to be accurate as we need to add up the number of
(2,4)-cycles of all sizes and bound it by $n_4^{n_4}$.\\
The decomposition of the euler tour into cycles is not unique. \\We need to be sure that some suitable number of steps of the blue walk no (2,4)-cycle exists in any decomposition.
 }

Thus  for $2$-cycles (case $b=0$) we have $\E C(i,i,0)=O(1/i)$.
If $b>0$ then for some $\b<1/7e$,
\[
\sum_{i< \b n_2/n_4} C(i,a,b)=o(1),
\]
and thus w.h.p. all $(2,4)$-cycles are size at least $\Th(n_2/n_4)$.
The expected number of all  $(2,4)$-cycles is
\[
\sum_{a \le n_2} \sum_{b \le n_4}\sum_{i \le n_2+n_4} \E C(i,a,b)=
O\brac{n_2^{3/2} n_4^{1/2}\; (c n_4)^{n_4}}.
\]
Thus w.h.p. the total number $L(t_1)$ of such cycles of all sizes is at most $L(t_1)=O(n_2^{3/2}(n_4)^{n_4+1} \log n)$.

Let $\cE(t_b)$ be the event that
\[
\cE(t_B)= \{\text{After $t_B+T_G$ further blue steps, there exists an unvisited $(2,4)$-cycle.}\}
\]
Let $t_B$ be given by
\[
t_B=  \frac{n_4}{n_2} K n \;\log n \log\log n  \le  n^{2/3}.
\]
Using \eqref{2lemmas}, conditional on $n_4 \le \om \log n$ and $\om \le \log \log n$,
for some $\a>0$ constant we have
\beq{no4}
\Pr(\cE(t_B))\leq \Th\brac{n_2^{2}n_4^{n_4+1}}e^{-\a\log n\log\log n}=o(1).
\eeq

\paragraph{Part 3.}
Let $t_2=dn(1-\d_2)$ be the red time reached by the edge-process
after the further $t_B$ blue steps made in Part 2 of the proof.
The precise value of $\d_2$ is unknown, but the vacant net $\wh\G(t_2)$
consists only of 2-cycles. The existence of a vertex of red degree 4
contradicts $\Pr(\cE(t_B))=o(1)$ in \eqref{no4}. Thus $\wh\G(t_2)$ is a random
2-regular graph. As $\wh\G(t_1)$ has $n_2+n_4=n_2 \ooi$ vertices of positive red degree, and $\wh\G(t_2)$
is a subgraph of $\wh\G(t_1)$, it also has at most this many vertices of red degree 2.
By the result for $\E C(i,i,0)=O(1/i)$ in Part 2,
 in expectation, $\wh\G(t_2)$ has $\E C(i)=O(1/i)$ cycles of length $i$.
Thus
\[
\Pr(\text{ There are more than } s^2 \E C(s) \text{ cycles size } s \text{
for any } s \ge \log n) \le \sum_{s \ge \log n} \frac{1}{s^2}=O\bfrac{1}{\log n}.
\]
Condition on the number of cycles size $s$ being at most $s^2 \E C(s)=O(s)$.
Using \eqref{2lemmas}, for some constant $\a>0$, the probability $P_s(t)$ that some cycle size $s$
remains unvisited after $t$ steps of the blue walk is
\[
P_s(t) = O\brac{s e^{-t \a \frac{s}{n}}}.
\]
Let $\cF$ be the event that some red cycle of size at least $\log n$
is unvisited after
\[
t_B=\frac{K}{\a}\frac{n}{\log n} \log \log n
\]
further blue steps. Thus for $K \ge 3$,
\begin{align*}
\Pr(\cF)&\leq \sum_{s \ge \log n} P_s(t_B)\\
&\leq o(1)+\sum_{s\geq \log n}s^3\exp\brac{-\frac{Ks\log\log n}{\log n}}\\
&=o(1).
\end{align*}
\end{proof}

\section{Acknowledgement}

Our particular thanks  to Gesine Reinert who suggested the problem of vacant nets to us, and who continued to encourage the
development of this paper.
We also thank the anonymous referees who, among other things, suggested we include the threshold results for the random walk which prefers unvisited edges.

\newpage

\section{Appendix} \label{Appx}

\subsection{Experimental results for the unvisited edge process}\label{expt}

\begin{figure*}[!htbp]
  \centering
 \includegraphics{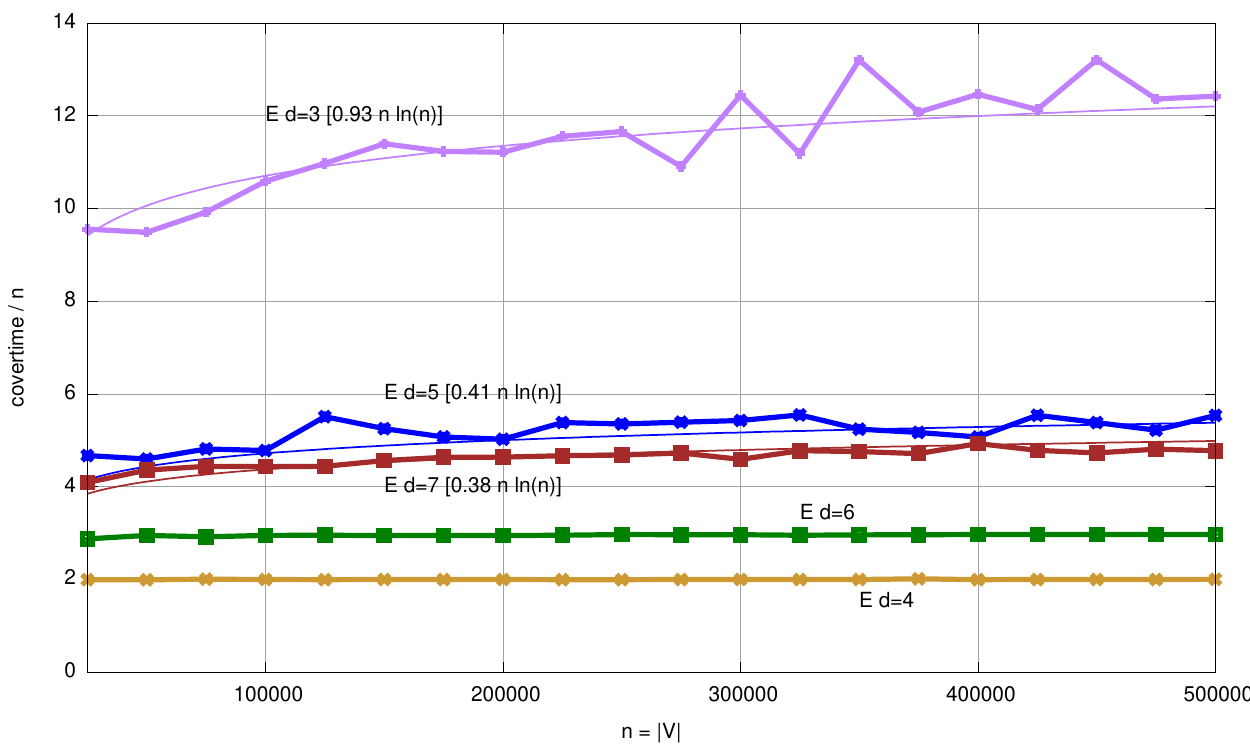}
   \caption{Vertex cover time of the unvisited edge process on $d$-regular graphs as function of  $n=|V|$. All cover times are normalized by dividing by the vertex set size $n$. The plot shows $d=3,4,5,6,7$. See \cite{BCF}  for  details.}
  \label{fig:covertime}
\end{figure*}

\subsection{Estimates of $R_v$ for nice vertices}\label{def-nice}
Recall the definition of $\cN$, the set of nice vertices of $G$ as given in Section \ref{props}.
For a nice vertex $v$, the following lemma relates the value of $R_v$ as given in \eqref{Rv}
to the probability of a first return to $v$
in the graph obtained by extending the subgraph $H$ of depth $\ell_1$
around  $v$ to an  infinite $r$--regular tree $\cT$ rooted at $v$.
Note that, we do not require the root $v$ of $\cT$ to have degree $r$.

\begin{lemma}\label{lemRv}\
Let $v$ be a  vertex  of degree $d(v) \ge 1$ whose subgraph $H$ to distance $\ell_1$
in a graph $G$ induces of a tree in which all vertices except $v$ have degree $r$. Then
\beq{Rvtree}
R_v=(1+\ole)\frac{1}{1-f}\qquad
\text{ where }f=\frac{1}{r-1},
\eeq
where  $f$ is the probability of a first return to $v$ in $\cT$,
the extension of $H$ to an infinite $r$-regular tree.
The $o(1)$ term in \eqref{Rvtree} is  $o(\log^{-K}n)$ for any positive constant $K$.
\end{lemma}
\proofstart
Let $H$ denote the subgraph of $G$ induced by the set of vertices at distance at most $\ell_1$ from $v$.
This is a tree and we can embed it into an infinite $r$-regular
tree $\cT$ rooted at $v$. Let $W_v$ be the walk on $G$ starting from $v$, and
let $\cX$ be  the walk on $\cT$,
starting from $v$.

Let $X_0=0$, a let $X_t$ be the distance of $\cX$ from
the root vertex $v$ at step $t$.
Let $D_0=0$, and let $D_t$ be the distance from $v$ of $W$ in $G$ at step $t$.
Note that we can couple $W_v,\cX$ so that $D_t=X_t$ up until the
first time that $D_t>\ell_1$.

The values of $X_t$ are as follows: $X_0=0,\; X_1=1$, and
if $X_t=0$ then $X_{t+1}=1$. If $X_t >0$ then
\begin{equation}\label{pqs}
X_t = \left\{
\begin{array}{ll}
X_{t-1}- 1& \text{ with probability } q=\frac{1}{r}\\
X_{t-1}+1 & \text{ with probability } p=\frac{r-1}{r}.
\end{array}
\right.
\end{equation}
The following result (see e.g. \cite{Feller}) is for  a
random walk on the line $=\set{0,...,a}$ with absorbing states
$\set{0,a}$, and transition probabilities $q,p$  for moves left and
right respectively. Starting at vertex $z$, the
probability of absorption at the origin 0 is
\begin{equation}\label{absorb}
\r(z,a)= \frac{ (q/p)^z-(q/p)^a}{1-(q/p)^a}\le \bfrac{q}{p}^z,
\end{equation}
provided $q \le p$.

Let $U_\infty=\set{\exists t\ge 1: X_t=0}$, i.e. the event that the particle ever
returns to the root vertex in $\cT$.
It follows from \eqref{absorb} with $z=1$ and $a=\infty$ that
\beq{inf}
f=\Pr(U_\infty)=\frac{1}{r-1}.
\eeq
It follows that the expected number of visits by $\cX$ to $v$ is
$$
\frac{1}{1-f}.
$$
We write
$$R_v=\sum_{t=0}^Tr_t\text{ and }\r=\sum_{t=0}^\infty\r_t$$
where $\r_t=\Pr(X_t=v)$.
Now $r_t=\r_t$ for $t\leq \ell_1$ and part (a) follows {once we prove that}
\beq{tails}
\sum_{t=\ell_1+1}^Tr_t=\ole\text{ and }\sum_{t=\ell_1+1}^\infty\r_t=\ole.
\eeq
The first equation of \eqref{tails} follows from
\beq{feq}
\card{r_t-\frac{1}{n}}\leq \l_{\max}^t
\eeq
where $\l_{\max}$ is the second largest eigenvalue of the walk. This follows from \eqref{mix}.

The second equation of \eqref{tails} is proved in Lemma 7 of \cite{CFReg} where it is shown that
\beq{catalan}
\sum_{t=\ell_1+1}^\infty\r_t\leq \sum_{2j=\ell_1+1}^\infty\binom{2j}{j}\frac{(r-1)^j}{r^{2j}}
\leq \sum_{2j=\ell_1+1}^\infty\bfrac{4(r-1)}{r^2}^j.
\eeq
{Thus
$$R_v=\r+O(T\l_{\max}^{\ell_1}+T/n+(8/9)^{\ell_1})$$}.
\proofend
\\
{\bf Remark.}
We can  use the method of Lemma \ref{lemRv} to calculate $R_u$ for a vertex $u=\g(S)$
in a graph $H$ obtained from $G$ by contracting a finite set of vertices $S$
to a single vertex $u=\g(S)$,
either directly, or after subdividing
sets of edges incident with these vertices. We assume that all vertices in
$S$ have a unique neighbour $w$ in $N(S)$, and that $w$ is tree-like to depth $\ell=\ell_1$
in $G-S$. It follows that, in $H$,
\[
R_u= \ooi \frac{1}{1-f_u},
\]
where $f_u$ is the probability of first return to $u$ in the graph $\cT(S)$ obtained
by extending the $r$-regular trees rooted at vertices of $N(S)$ to infinity,
and then contracting $S$ to $u=\g(S)$.

\subsection{Mixing time of chain $\cM$}
\label{snotproof}

\begin{lemma}
For $G \in G_r$, $r \ge 3$ constant, w.h.p. $T_{\cM}=O(\log n)$.
\end{lemma}
\begin{proof}
Mihail \cite{MM} gives the following conductance based measure of convergence for a strongly aperiodic walk with transition matrix $P$ on a  $d$-regular digraph $D=(V,A)$. For vertices $e,f \in V$,
\beq{Mi}
|P_e^{t}(f)-\pi_f| \le (1-\a^2)^{t/2}.
\eeq
Here,
\[
\a= \frac{1}{2d} \min_{|B| \le |V|/2} \frac{|C(B)|}{|B|},
\]
and $B \seq V$ and $C(B)=\{a \in A: a=(e,f), e \in B, f \in \ol B\}$. The proof in \cite{MM} assumes the walk is lazy (i.e. for our model the non-backtracking walk on the underlying graph is lazy).

In Lemma \ref{fruitsbasket} (below) we prove  there is an $\e>0$ constant such that w.h.p.
$\a \ge \e/4r$. The result that $T_{\cM}=O(\log n)$ follows from using this in \eqref{Mi}.
\end{proof}

\begin{lemma}\label{fruitsbasket}
For $G \in G_r$, there is an $\e>0$ constant such that w.h.p.
$\a \ge \e/4r$.
\end{lemma}
\begin{proof}
For our chain $\cM$, $d=r-1$, and $|V_M|=rn$ where $V_M$ is the set of oriented arcs of the
underlying graph $G$. Suppose that $B \seq V_M$ is a set of vertices of $\cM$ (directed arcs of $G$).
Let $R=\ol B=V_M-B$ denote the rest of the arcs.
 Thus $|B|+|R|=rn$. Assume that $|R|\geq |B|$.
 We need to estimate $C(B)$. For a vertex $v \in V(G)$, let $d_R^+(v)$ etc. be the $R$-out-degree of $v$ (i.e. $d^+_R(v)=|\{\{v,w\} \in E(G):(v,w) \in R\}|$). Next let
\begin{align*}
W_0&=\set{w:d^+_B(w)=r-1\text{ and }d^-_B(w)=1},\\
W_{1,s}&=\set{w:d^+_B(w)=r,\,d^-_B(w)=s}\text{ and }W_1=\bigcup_{s=0}^rW_{1,s}.
\end{align*}
If $(v,w) \in B$ and  $w \not \in W_0 \cup W_1$ there is always an edge $(w,x)$, $x \ne v$ such that $(w,x) \in R$. If $e=(v,w)\in B$ and $f=(w,x)\in R, \; x\neq v$ then the transition from $e$ to $f$ is non-backtracking, and arc $(e,f)$ contributes to $C(B)$. We can bound $|C(B)|$ from below by
\[
|C(B)| \ge \sum_{(v,w)\in B}\brac{1-1_{w\in W_0\cup W_1}}.
\]
Enumerating $W_0\cup W_1$ by in-degree gives
\begin{align}
\sum_{(v,w)\in B}\brac{1-1_{w\in W_0\cup W_1}}&=|B|-|W_0|-\sum_{s=0}^r\sum_{w\in W_{1,s}}d_B^-(w)\nonumber\\
&=|B|-|W_0|-\sum_{s=0}^rs|W_{1,s}|\label{eq1}.
\end{align}
Enumerating $B$ by initial and terminal vertices gives
$$\sum_{s=0}^r(r+s)|W_{1,s}|+r|W_0|\leq 2|B|.$$
So,
\begin{align*}
\sum_{(v,w)\in B}\brac{1-1_{w\in W_0\cup W_1}}&\geq \frac12\sum_{s=0}^r(r+s)|W_{1,s}|+\frac{r}{2}|W_0|-|W_0|-\sum_{s=0}^rs|W_{1,s}|\\
&=\sum_{s=0}^r\brac{\frac{r}{2}-\frac{s}{2}}|W_{1,s}|+\brac{\frac{r}{2}-1}|W_0|.
\end{align*}
{\bf Case 1:} $\exists \;  0\leq s<r$ such that $|W_{1,s}|\geq \e |B|$ or $|W_0|\geq \e |B|$.

In this case,
$$\sum_{(v,w)\in B}\brac{1-1_{w\in W_0\cup W_1}}\geq \frac{\e}{2}|B|.$$
{\bf Case 2:} $|W_{1,s}|< \e |B|,\,\forall 0\leq s<r$ and $|W_0|< \e |B|$ and $|W_{1,r}|\leq r^{-1}\brac{1-\frac{r^2}{2}\e}|B|$.

Going back to \eqref{eq1} we get
$$\sum_{(v,w)\in B}\brac{1-1_{w\in W_0\cup W_1}}\geq |B|\brac{1-\e-\frac{r(r-1)}{2}\e-\brac{1-\frac{r^2}{2}\e}}=\frac{\e}{2}|B|.$$

{\bf Case 3:} $|W_0|< \e |B|$, $|W_{1}|> r^{-1}\brac{1-\frac{r^2}{2}\e}|B|$ and $|W_{1}|\leq \frac34n$.

Let $e(W_1, \ol W_1)$ be the number of edges between $W_1$ and $\ol W_1$ in the underlying graph $G$, and let $\F=\F(G)$ be the conductance of $G$. Thus
\[
e(W_1, \ol W_1) \ge \min(|W_1|, | \ol W_1|) \; r \F.
\]
If $u \in W_1$,  and $\{u,v\}$ is an edge of $G$, then by definition of $W_1$, the arc $(u,v) \in B$. Thus if $v \in \ol W_1$, and $v \not \in W_0$, there is some $z \in V$, $z \ne u$ such that $(v,z) \in R$. Let $A$ be the set of such good arcs $(v,z)$, then
\[
|C(B)|\geq |A| \ge e(W_1, \ol W_1)- |W_0|.
\]
If $|W_1|\le n/2$,
\[
|A| \; \ge \; |W_1| r \F-|W_0| \;\ge \;((1-r^2\e/2)\F-\e)|B|.
\]
If $n/2 \le |W_1| \le 3n/4$, and $|B| \le rn/2$,
\[
|A| \; \ge \; | \ol W_1| r \F-|W_0| \;\ge\; \frac{n}{4}r\F-\e |B|\;\ge\; (\F /2-\e )|B|.
\]
In either case, for  $r^2\e<1$,  $|C(B)| \ge |B|(\F /2-\e )$.

{\bf Case 4:}
$|W_{1}|> \frac34n$.

If $|B| \le rn/2$, this  is impossible since we have $|B|\geq r|W_{1}|>\frac{3}4 rn$.
\end{proof}

\end{document}